\documentclass[reqno]{amsart}

\usepackage{mathrsfs}
\usepackage{amssymb}
\usepackage{cite,comment}
\usepackage{latexsym}
\usepackage{graphicx}

\usepackage{color}
\usepackage{comment}

\theoremstyle{plain}
\newtheorem{thm}{Theorem}[section]
\newtheorem{lem}[thm]{Lemma}
\newtheorem{cor}[thm]{Corollary}

\newtheorem{prop}[thm]{Proposition}
\newtheorem{rmk}[thm]{Remark}

\def\D{\mathrm{D}}

\def\U{\mathscr{U}}

\def\c{\mathrm{c}}
\def\d{\mathrm{d}}

\def\Nset{\mathbb{N}}
\def\Rset{\mathbb{R}}
\def\Sset{\mathbb{S}}
\def\Tset{\mathbb{T}}
\def\Zset{\mathbb{Z}}

\def\epsilon{\varepsilon}

\makeatletter
 \@addtoreset{equation}{section}
\makeatother
\def\theequation{\arabic{section}.\arabic{equation}}

\begin{document}


\title[Synchronized solutions in the Kuramoto model]
{Bifurcations and stability of synchronized solutions in the Kuramoto model
 with uniformly spaced natural frequencies}

\author{Kazuyuki Yagasaki}

\address{Department of Applied Mathematics and Physics, Graduate School of Informatics,
Kyoto University, Yoshida-Honmachi, Sakyo-ku, Kyoto 606-8501, JAPAN}
\email{yagasaki@amp.i.kyoto-u.ac.jp}

\date{\today}

\begin{abstract}
We consider the classical Kuramoto model (KM) with natural frequencies
 and its continuum limit (CL),
 and discuss the existence of synchronized solutions and their bifurcations and stability.
We specifically assume that
 the frequency function is symmetric and linear in the CL,
 so that the natural frequencies are evenly spaced in the KM.
We show that in the KM,
 $O(2^n)$ one-parameter families of synchronized solutions are born
 and $O(2^n)$ {saddle-node and} pitchfork bifurcations occur at least,
 when the node number $n$ is odd and tends to infinity.
Moreover, we prove that
 a family of synchronized solutions obtained in the previous work
 suffers a saddle-node bifurcation at which its stability changes
 from asymptotically stable  to unstable
 and the other families of synchronized solutions are unstable in the KM.
For the CL, we show that
{\color{black}
 a one-parameter family of continuous synchronized solutions obtained in the previous work
 is asymptotically stable
 and that there exist uncountably many one-parameter families of discontinuous synchronized solutions
 and they are unstable along with another one-parameter family of continuous synchronized solutions.}
\end{abstract}
\maketitle


\section{Introduction}

We consider the classical Kuramoto model (KM) \cite{K75,K84} with natural frequencies,
\begin{equation} 
  \frac{\d}{\d t} u_i^n (t) =\omega_i^n 
  +\frac{K}{n} \sum^{n}_{j=1}\sin \left( u_j^n(t) - u_i^n(t) \right),\quad
   i\in [n]:=\{1,2,\ldots,n\}, 
\label{eqn:dsys} 
\end{equation}
where $u_i^n(t)\in\Sset^1=\Rset/2\pi\Zset$ and  $\omega_i^n\in\Rset$
 stand for the phase and natural frequency of the oscillator at the node $i\in[n]$ 
 and $K$ is a coupling constant.
The model \eqref{eqn:dsys} and its generalizations have provided mathematical models
 of various problems arising in many fields
 including physics, biology, chemistry, engineering, and economics,
 and have extensively been studied.
See \cite{S00,PRK01,ABVRS05,ADKMZ08,DB14,PR15,RPJK16}
 for the reviews of the enormous research works.

In the previous work \cite{IY23},
 coupled oscillator networks including \eqref{eqn:dsys} were studied
 and shown to be well approximated by the corresponding continuum limits (CLs),
 for instance, which are given by
\begin{equation} 
\frac{\partial}{\partial t}u(t,x)=\omega(x)+K\int _{I}\sin(u(t,y)-u(t,x))\d y
\label{eqn:csys}
\end{equation}
for \eqref{eqn:dsys} with
\begin{equation}
\omega_i^n=n\int_{I_i^n}\omega(x)dx, \quad i \in [n],
\label{eqn:omega}
\end{equation}
where $I=[0,1]$ and
\[
I_i^n:=
\begin{cases}
  [(i-1)/n,i/n) & \mbox{for $i<n$};\\
  [(n-1)/n,1] & \mbox{for $i=n$}.
\end{cases}
\]
We call $\omega(x)\in L^2(I)$ a \emph{frequency function}.
More general cases in which networks of coupled oscillators are defined 
 on multiple graphs which may be deterministic or random, and dense or sparse
 were discussed in \cite{IY23}.
Similar results for such networks which are defined on single graphs
 and do not have natural frequencies depending on each node
 were obtained earlier in \cite{KM17,M14a,M14b,M19}
 although they are not applicable to \eqref{eqn:dsys} and \eqref{eqn:csys}.
The same CL as \eqref{eqn:csys} was also adopted for the KM \eqref{eqn:dsys}
 without a rigorous mathematical guarantee much earlier in \cite{E85},
 while similar CLs were utilized
 for the KM with nonlocal coupling and a single or zero natural frequency
 in \cite{AS06,GHM12,WSG06}.

A different approach for approximation of coupled oscillators like \eqref{eqn:dsys}
 by integro-partial differential equations called the \emph{Vlasov equations}
 was more frequently used although its mathematical foundations were provided recently
 in \cite{C15,CN11} (see also \cite{CM19a,CM19b} for further extensions).
See \cite{S00,ABVRS05,DB14,RPJK16,DF18} and references therein for its applications.
Compared with those results, where probability density functions are treated,
 two advantages of the approach of \cite{IY23}
 are to deal with a complete deterministic case
 in which the natural frequencies are fixed like \eqref{eqn:dsys},
 and to give more direct description on the dynamics of the coupled oscillators
 without using probability density functions.

Moreover, it was shown in \cite{IY23} that
 the CL \eqref{eqn:csys} has the synchronized solutions
\begin{equation}
u(t,x)=U(x)+\Omega t+\theta,\quad
U(x)=\arcsin\left(\frac{\omega(x)-\Omega}{KC}\right),
\label{eqn:csol}
\end{equation}
where $\theta\in\Sset^1$ is an arbitrary constant,
 the range of the function $\arcsin$ is $[-\tfrac{1}{2}\pi,\tfrac{1}{2}\pi]$ and
\[
\Omega=\int_I \omega(x)\d x,
\]
if there exists a constant $C>0$ such that
\begin{equation}
C=\int_I \sqrt{1-\left(\frac{\omega(x)-\Omega}{KC}\right)^2}\d x.
\label{eqn:CU}
\end{equation}
This is easily confirmed by substituting \eqref{eqn:csol} into \eqref{eqn:csys}.
A solution of the form \eqref{eqn:csol} to \eqref{eqn:csys} was also obtained in \cite{E85}
 although the constant term 
 was not contained.
Similarly, the KM \eqref{eqn:dsys} is shown to have the synchronized solutions
\begin{equation}
u_i^n(t) = \Omega_\D t + U_i + \theta,\quad
U_i=\arcsin\left(\frac{\omega_i^n - \Omega_\D}{K C_\D}\right),
\label{eqn:dsol}
\end{equation}
where $\theta\in\Sset^1$ is an arbitrary constant and
\[
\Omega_\D=\frac{1}{n}\sum_{i=1}^n \omega_i^n,
\]
if there exists a constant $C_\D>0$ such that
\begin{equation}
C_\D=\frac{1}{n}\sum_{i=1}^n\sqrt{1-\left(\frac{\omega_i^n - \Omega_\D}{K C_\D}\right)^2}.
\label{eqn:CD}
\end{equation}
This is easily checked by substituting \eqref{eqn:dsol} into \eqref{eqn:dsys}.
We also see that
\begin{equation}
U_i\to U(x),\quad
C_\D\to C\quad
\mbox{as $n\to\infty$},
\label{eqn:lim}
\end{equation}
where $i/n\to x$ as $n\to\infty$.
See Section 3.1 of \cite{IY23} for more details.

\begin{figure}
\includegraphics[scale=0.5]{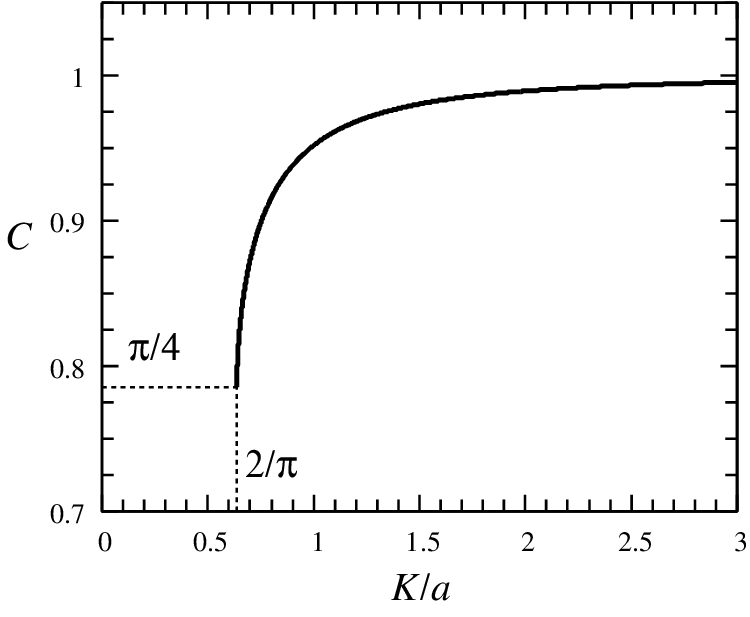}
\caption{Dependence of $C$ on $K/a$ in \eqref{eqn:C}.
\label{fig:1a}}
\end{figure}

We specifically consider the case in which
 the frequency function $\omega(x)$ is symmetric and linear, i.e.,
\begin{equation}
\omega(x)=a(x-\tfrac{1}{2}),
\label{eqn:omegaex}
\end{equation}
where $a>0$ is a constant.
Then the natural frequencies are placed equally, and
\[
\omega_i^n=\frac{a}{2n}(2i-n-1),
i\in[n].
\]
We have $\Omega=0$ and compute \eqref{eqn:CU} as
\begin{equation}
C=\frac{KC}{a}\left(\arcsin\left(\frac{a}{2KC}\right)
 +\frac{a}{2KC}\sqrt{1-\left(\frac{a}{2KC}\right)^2}\right),
\label{eqn:Cex}
\end{equation}
which yields
\begin{equation}
\frac{a}{K}=\varphi\left(\frac{a}{2KC}\right),
\label{eqn:C}
\end{equation}
where $\varphi(\eta)=\arcsin\eta+\eta\sqrt{1-\eta^2}$.
Since $\varphi(\eta)$ is monotonically increasing on $(0,1)$ and
\[
\varphi(0)=0,\quad \varphi(1)=\tfrac{1}{2}\pi,
\]
the synchronized solutions \eqref{eqn:csol} with $\Omega=0$ exist
 if {\color{black}and only if} $K/a\ge 2/\pi$,
 and it is continuous in $x$.
The dependence of $C$ on $a/K$ which is calculated from \eqref{eqn:C}
 is displayed in Fig.~\ref{fig:1a}.
Hence, the synchronized solutions \eqref{eqn:csol} suddenly appear,
 i.e., a {``bifurcation''} occurs, at $K=K_\c:=2a/\pi$
 when the coupling constant $K$ is taken as the control parameter
 and it is increased, say, from zero.
Moreover, for $n>0$ sufficiently large,
 the synchronized solutions \eqref{eqn:dsol} with $\Omega_\D=0$ suddenly appear near $K=K_\c$
 in the KM \eqref{eqn:dsys},
 which is well approximated by the CL \eqref{eqn:csys} as stated above, 
 when the coupling constant $K$ is increased.
See Section 3.2 of \cite{IY23} for the details.

The above observation seems strange
 from the viewpoint of dynamical systems theory \cite{GH83,W03}.
For finite-dimensional systems (and most infinite-dimensional systems),
 if an equilibrium state suddenly appears when a control parameter is changed,
 then a saddle-node bifurcation yielding a pair of equilibria
 of different stability types generally occurs.
 according to the theory.
So we conjecture that such a bifurcation occurs in the KM \eqref{eqn:dsys}
 when the synchronized solutions \eqref{eqn:dsol} appear.
On the other hand, in the CL \eqref{eqn:csys},
 no other continuous synchronized solutions are
 {\color{black}born from the same ones as \eqref{eqn:csol} at $K/a=2/\pi$}
 (see Theorem~\ref{thm:7a}).
One of our main objects is to explain the reason
 why such a phenomenon occurs in the CL \eqref{eqn:csys}.

In this paper, we discuss synchronized solutions and their bifurcations and stability
 in the KM \eqref{eqn:dsys} whose natural frequencies correspond
 to the symmetric linear frequency function \eqref{eqn:omegaex},
 so that they are evenly spaced.
Our discussions can apply to asymmetric linear frequency functions
 after the system \eqref{eqn:dsys} is transformed
 in the rotational frame with the rotational speed $\Omega_\D$.
The approaches used here can be extended to more general cases of natural frequencies
 although we restrict ourselves to such a simple case.

The KM \eqref{eqn:dsys} with a finite node number $n$ have been studied
 in many but fewer references than the infinite node number case.
The critical values {of $K$} for the existence and stability of synchronized solutions
 were investigated in \cite{DB11,MS05,OS16,P05,VW93,VM08}.
In \cite{OS16,P05},
 it was assumed that the natural frequencies are evenly spaced as in our setting
 (and additionally $n$ is odd in \cite{OS16}).
Relationships between the natural frequencies and synchronization
 were studied in \cite{BCD21,DB11},
 and complete synchronization states were discussed in \cite{CHJK12,DX13,HKP15}
 for a large coupling strength.
The occurrence of partial synchronization was also argued in \cite{AR04,BW21,HR20}.
For the cases of small node numbers of $n\le 5$, 
 bifurcations and chaotic attractors were detected and found mainly numerically
 in \cite{MPBT04,MPT05a,MPT05b}.
 
Here we assume that the node number $n$ is odd as in \cite{OS16},
 and show that $O(2^n)$ families of equilibria are born
 and $O(2^n)$ {saddle-node and} pitchfork bifurcations
 occur at least as $n\to\infty$
 (see Remark~\ref{rmk:5a}(iii) and Proposition~\ref{prop:5a}
 for more precise statements).
Moreover, we prove that
 the family of synchronized solutions \eqref{eqn:dsol}
 suffers a saddle-node bifurcation at which its stability changes
 from asymptotically stable  to unstable
 and the other families of synchronized solutions are unstable.
For the CL \eqref{eqn:csys}, we show that
 {\color{black}
 the one-parameter family of continuous synchronized solutions \eqref{eqn:csol} 
 is asymptotically stable
 and that there exist uncountably many one-parameter families of discontinuous synchronized solutions
 and they are unstable along with another one-parameter family of continuous synchronized solutions.}
 
The outline of this paper is as follow:
In Section 2, the previous results of \cite{IY23} are reviewed
 in the context of the KM \eqref{eqn:dsys} and its CL \eqref{eqn:csys}.
We also expand the previous result on a relationship of their stability
 and give new results on relationships of their instability.
After analyzing the case of $n=3$ in Section~3,
 we provide our results for the existence of synchronized solutions in Section~4,
 and their bifurcations and stability in Sections~5 and 6, respectively.
Finally, we discuss the implications of the results in Sections~4-6
 for the CL \eqref{eqn:csys} in Section~7,
 based on the results of Section~2.
We extensively use dynamical systems theory throughout the paper,
 and recommend the readers to consult the textbooks \cite{GH83,W03}
 if they are not familiar with the theory.
 

\section{Previous and New Fundamental Results}

We first review the results of \cite{IY23} in the context of  \eqref{eqn:dsys} and \eqref{eqn:csys}
 {and give new fundamental results
  on relationships between the KM \eqref{eqn:dsys} and CL \eqref{eqn:csys},
 which can be extended to more general coupled oscillator networks
 as in \cite{IY23}.}
See Section~2 and Appendices~A and B of \cite{IY23} for more details
 including the proofs of the theorems stated below
 except for Theorems~\ref{thm:2c}(ii), \ref{thm:2d} and \ref{thm:2e}.

Let $g(x)\in L^2(I)$.
We have the following on the existence and  uniqueness of solutions
 to the initial value problem (IVP) of the CL \eqref{eqn:csys}
 (see Theorem~2.1 of \cite{IY23}).
 
\begin{thm}
\label{thm:2a}
{For any $T\in(0,\infty)$,
 there exists a unique solution $\mathbf{u}(t)\in\linebreak C^1([0,T],L^2(I))$}
 to the IVP of \eqref{eqn:csys} with
\[
u(0,x)=g(x).
\]
Moreover, the solution depends continuously on $g$.
\end{thm}

We next consider the IVP of the KM \eqref{eqn:dsys}
 and turn to the issue on convergence of solutions in \eqref{eqn:dsys}
 to those in the CL \eqref{eqn:csys}.
Since the right-hand side of \eqref{eqn:dsys} is Lipschitz continuous in $u_i^n$, $i\in[n]$,
 we see by a fundamental result of ordinary differential equations
 (e.g., Theorem~2.1 of Chapter~1 of \cite{CL55})
 that the IVP of \eqref{eqn:dsys} has a unique solution.
Let $\mathbf{u}:\Rset\to L^2(I)$ stand for an $L^2(I)$-valued function.
Given a solution $u_n(t)=(u_1^n(t), u_2^n(t), \ldots, u_n^n(t))$ to the IVP of \eqref{eqn:dsys},
 we define an $L^2(I)$-valued function $\mathbf{u}_n:\Rset\to L^2(I)$ as
\begin{equation*}
\mathbf{u}_n(t) = \sum^{n}_{i=1} u_i^n(t) \mathbf{1}_{I_i^n},
\end{equation*}
where {$\mathbf{1}_{I_i^n}$}
 represents the characteristic function of $I_i^n$ for $i\in[n]$.
Let $\|\cdot\|$ denote the norm in $L^2(I)$.
We have the following from Theorem~2.3 of\cite{IY23}.

\begin{thm}
\label{thm:2b}
If $\mathbf{u}_{n}(t)$ is the solution to the IVP of \eqref{eqn:dsys} such that
\begin{equation}
\lim_{n\to\infty}\|\mathbf{u}_n(0)-\mathbf{u}(0)\|=0,
\label{eqn:icd}
\end{equation}
then for any $T > 0$ we have
\[
\lim_{n \rightarrow \infty}\max_{t\in[0,T]}\|\mathbf{u}_n(t)-\mathbf{u}(t)\|=0,
\]
where $\mathbf{u}(t)$ represents the solution
 to the IVP of the CL \eqref{eqn:csys}. 
\end{thm}

\begin{rmk}
The condition~\eqref{eqn:icd} holds if one takes
\begin{equation}
u_i^n(0) = u_{i0}^n:=n\int_{I_i^n}{u(0,x)}\d x
\label{eqn:icd0}
\end{equation}
for $n>0$ sufficiently large.
Equation~\eqref{eqn:icd0} was assumed instead of \eqref{eqn:icd}
 in the original statement of Theorem~$2.3$ in {\rm\cite{IY23}},
 which can be modified as in Theorem~$\ref{thm:2b}$.
See Appendix~B of {\rm\cite{IY23}}.
\end{rmk}

For $\theta\in\Sset^1$,
 let $\boldsymbol{\theta}$ represent the constant function $u=\theta$.
If $\bar{\mathbf{u}}_n(t)$ is a solution to the KM \eqref{eqn:dsys},
 then so is $\bar{\mathbf{u}}_n(t)+\boldsymbol{\theta}$ for any $\theta\in\Sset^1$.
Similarly, if $\bar{\mathbf{u}}(t)$ is a solution to the CL \eqref{eqn:csys},
 then so is $\bar{\mathbf{u}}(t)+\boldsymbol{\theta}$ for any $\theta\in\Sset^1$.
Let $\U_n=\{\bar{\mathbf{u}}_n(t)+\boldsymbol{\theta}\mid\theta\in\Sset^1\}$
 and $\U=\{\bar{\mathbf{u}}(t)+\boldsymbol{\theta}\mid\theta\in\Sset^1\}$
 denote the families of solutions to \eqref{eqn:dsys} and \eqref{eqn:csys}
 like \eqref{eqn:dsol} and \eqref{eqn:csol}, respectively.
We say that $\U_n$ (resp. $\U$) is \emph{stable}
 if solutions starting in its (smaller) neighborhood
 remain in its (larger) neighborhood for $t\ge 0$,
 and \emph{asymptotically stable} if $\U_n$ (resp. $\U$) is stable
 and the distance between such solutions and $\U_n$ (resp. $\U$) converges to zero
 as $t\to\infty$.
We obtain the following result.

\begin{thm}
\label{thm:2c}
Suppose that the KM\eqref{eqn:dsys} and CL \eqref{eqn:csys}
 have solutions $\bar{\mathbf{u}}_n(t)$ and $\bar{\mathbf{u}}(t)$, respectively, such that
\begin{equation}
\lim_{n\to\infty}\|\bar{\mathbf{u}}_n(t)-\bar{\mathbf{u}}(t)\|=0
\label{eqn:2c}
\end{equation}
for any $t\in[0,\infty)$.
Then the following hold$:$
\begin{enumerate}
\setlength{\leftskip}{-1.8em}
\item[\rm(i)]
If $\U_n$ is stable $($resp. asymptotically stable$)$ for $n>0$ sufficiently large,
 then $\U$ is also stable $($resp. asymptotically stable$)$.
\item[\rm(ii)]

If $\U$ is stable, then for any $\epsilon,T>0$ there exists $\delta>0$
 such that for $n>0$ sufficiently large,
 if $\mathbf{u}_n(t)$ is a solution to the KM \eqref{eqn:dsys} satisfying
\begin{equation}
\min_{\theta\in\Sset^1}
 \|\mathbf{u}_n(0)-\bar{\mathbf{u}}_n(0)-\boldsymbol\theta\|<\delta,
\label{eqn:2c1}
\end{equation}
then
\begin{equation}
\max_{t\in[0,T]}\min_{\theta\in\Sset^1}
 \|\mathbf{u}_n(t)-\bar{\mathbf{u}}_n(t)-\boldsymbol\theta\|<\epsilon.
\label{eqn:2c2}
\end{equation}
Moreover, if $\U$ is asymptotically stable, then
\begin{equation}
\lim_{t\to\infty}\lim_{n\to\infty}{\color{black}\min_{\theta\in\Sset^1}}
 \|\mathbf{u}_n(t)-\bar{\mathbf{u}}_n(t)-\boldsymbol\theta\|=0,
\label{eqn:2c3}
\end{equation}
where $\mathbf{u}_n(t)$ is any solution to \eqref{eqn:dsys}
 such that $\mathbf{u}_n(0)$ is contained in the basin of attraction for $\U$.
\end{enumerate}
\end{thm}

\begin{proof}
We only give a proof of part~(ii)
 since part~(i) is obtained from Theorem~2.7(i) in \cite{IY23}.
Suppose that $\U$ is stable
 and let $\epsilon>0$ and $T>0$ be sufficiently small and large, respectively.
Then there exists $\delta>0$ such that
 if $\mathbf{u}(t)$ is a solution to the CL \eqref{eqn:csys} satisfying
\[
{\color{black}\min_{\theta\in\Sset^1}}
 \|\mathbf{u}(0)-\bar{\mathbf{u}}(0)-\boldsymbol\theta\|
 <3\delta<\tfrac{1}{2}\epsilon,
\]
then
\[
{\color{black}\min_{\theta\in\Sset^1}}\|\mathbf{u}(t)-\bar{\mathbf{u}}(t)-\boldsymbol\theta\|
<\epsilon-2\delta
\]
for any $t\in[0,T]$.
By Theorem~\ref{thm:2b} and our assumption,
 we can take {\color{black}a solution $\mathbf{u}(t)$ to the CL \eqref{eqn:csys} and}
 $N>0$ sufficiently large such that if $n>N$, then
\[
\|\mathbf{u}_n(t)-\mathbf{u}(t)\|,
\|\bar{\mathbf{u}}_n(t)-\bar{\mathbf{u}}(t)\|<\delta.
\]
Hence, if Eq.~\eqref{eqn:2c1} holds and $n>N$, then
\begin{align*}
{\color{black}\min_{\theta\in\Sset^1}}\|\mathbf{u}(0)-\bar{\mathbf{u}}(0)-\boldsymbol\theta\|
<&{\color{black}\min_{\theta\in\Sset^1}}\|\mathbf{u}_n(0)-\bar{\mathbf{u}}_n(0)-\boldsymbol\theta\|\\
& +\|\mathbf{u}_n(0)-\mathbf{u}(0)\|+\|\bar{\mathbf{u}}_n(0)-\bar{\mathbf{u}}(0)\|<3\delta,
\end{align*}
so that
\begin{align*}
{\color{black}\min_{\theta\in\Sset^1}}
 \|\mathbf{u}_n(t)-\bar{\mathbf{u}}_n(t)-\boldsymbol\theta\|
<&{\color{black}\min_{\theta\in\Sset^1}}
 \|\mathbf{u}(t)-\bar{\mathbf{u}}(t)-\boldsymbol\theta\|\notag\\
& +\|\mathbf{u}_n(t)-\mathbf{u}(t)\|+\|\bar{\mathbf{u}}_n(t)-\bar{\mathbf{u}}(t)\|
 <\epsilon,
\end{align*}
which yields \eqref{eqn:2c2}.
Moreover, if $\U$ is asymptotically stable, then
\[
\lim_{n\to\infty}{\color{black}\min_{\theta\in\Sset^1}}
 \|\mathbf{u}_n(t)-\bar{\mathbf{u}}_n(t)-\boldsymbol\theta\|
 ={\color{black}\min_{\theta\in\Sset^1}}\|\mathbf{u}(t)-\bar{\mathbf{u}}(t)-\boldsymbol\theta\|,
\]
in which we take the limit as $t\to\infty$ to obtain \eqref{eqn:2c3}.
\end{proof}

\begin{rmk}
$\U_n$ may not be stable or asymptotically stable in the KM \eqref{eqn:dsys}
 for $n>0$ sufficiently large even if so is $\U$ in the CL \eqref{eqn:csys}.
In the definition of stability and asymptotic stability of solutions to the CL \eqref{eqn:csys},
 we cannot distinguish two solutions that are different only in a set with the Lebesgue measure zero.
\end{rmk}

From the proof of Theorem~{\rm\ref{thm:2c}}
 we immediately the following without assuming the existence
 of the solution $\bar{\mathbf{u}}_n(t)$ to the KM \eqref{eqn:dsys}
 satisfying \eqref{eqn:2c}.

\begin{cor}
\label{cor:2a}
Suppose that the CL \eqref{eqn:csys} has a solution $\bar{\mathbf{u}}(t)$
 and $\U=\{\bar{\mathbf{u}}(t)+\boldsymbol{\theta}\mid\theta\in\Sset^1\}$ is stable.
Then for any $\epsilon,T>0$ there exists $\delta>0$ such that for $n>0$ sufficiently large,
 if $\mathbf{u}_n(t)$ is a solution to the KM \eqref{eqn:dsys} satisfying
\[
\min_{\theta\in\Sset^1}\|\mathbf{u}_n(0)-\bar{\mathbf{u}}(0)-\boldsymbol\theta\|<\delta,
\]
then
\[
\max_{t\in[0,T]}\min_{\theta\in\Sset^1}
 \|\mathbf{u}_n(t)-\bar{\mathbf{u}}(t)-\boldsymbol\theta\|<\epsilon.
\]
Moreover, if $\U$ is asymptotically stable, then
\[
\lim_{t\to\infty}\lim_{n\to\infty}{\color{black}\min_{\theta\in\Sset^1}}
 \|\mathbf{u}_n(t)-\bar{\mathbf{u}}(t)-\boldsymbol\theta\|=0,
\]
where $\mathbf{u}_n(t)$ is any solution to \eqref{eqn:dsys}
 such that $\mathbf{u}_n(0)$ is contained in the basin of attraction for $\U$.
\end{cor}

{\color{black}
Corollary~\ref{cor:2a} says that if $\U$ is (asymptotically) stable,
 then it behaves as if it is an (asymptotically) stable family of solutions in the KM \eqref{eqn:dsys}.}
We prove the following theorem.
 
\begin{thm}
\label{thm:2d}
Suppose that the hypothesis of Theorem~$\ref{thm:2c}$ holds.
Then the following hold$:$
\begin{enumerate}
\setlength{\leftskip}{-1.8em}
\item[\rm(i)]
If $\U_n$ is unstable for $n>0$ sufficiently large
 and {\color{black}
 no stable family of solutions to the KM \eqref{eqn:dsys}
 converges to $\U$ as $n\to\infty$,} then $\U$ is unstable$;$
\item[\rm(ii)]
If $\U$ is unstable, then so is $\U_n$ for $n>0$ sufficiently large.
\end{enumerate}
\end{thm}

\begin{proof}
Part~(ii) follows from the contraposition of Theorem~\ref{thm:2c}(i).
So we only give a proof part (i).

Suppose that $\U_n$ is unstable {\color{black}for $n>0$ sufficiently large
 and no stable family of solutions to the KM \eqref{eqn:dsys}
 converges to $\U$ as $n\to\infty$.}
Then for any $\epsilon,\delta,T>0$
 there exists a solution $\mathbf{u}_n(t)$ to the KM \eqref{eqn:dsys} such that
\[
\|\mathbf{u}_n(0)-\bar{\mathbf{u}}_n(0)\|<\tfrac{1}{3}\delta
\]
and
\begin{equation}
{\color{black}\min_{\theta\in\Sset^1}}
 \|\mathbf{u}_n(T)-\bar{\mathbf{u}}_n(T)-\boldsymbol\theta\|>3\epsilon.
\label{eqn:thm2d}
\end{equation}
Let $\mathbf{u}(t)$ be a solution to the CL \eqref{eqn:csys} with
\[
\|\mathbf{u}(0)-\mathbf{u}_n(0)\|<\tfrac{1}{3}\delta.
\]
By Theorem~\ref{thm:2b},
 for any $\epsilon>0$ there exists an integer $N>0$ such that
\[
\|\mathbf{u}(T)-\mathbf{u}_n(T)\|<\epsilon
\]
for $n>N$.
Moreover, by our assumption, we have
\[
\|\bar{\mathbf{u}}(0)-\bar{\mathbf{u}}_n(0)\|,
\|\bar{\mathbf{u}}(T)-\bar{\mathbf{u}}_n(T)\|<\min(\epsilon,\tfrac{1}{3}\delta)
\]
for $n>0$ sufficiently large.
Hence,
\[
\|\mathbf{u}(0)-\bar{\mathbf{u}}(0)\|
<\|\mathbf{u}(0)-\mathbf{u}_n(0)\|+\|\mathbf{u}_n(0)-\bar{\mathbf{u}}_n(0)\|
+\|\bar{\mathbf{u}}_n(0)-\bar{\mathbf{u}}(0)\|<\delta
\]
and
\begin{align*}
&
\|\mathbf{u}(T)-\bar{\mathbf{u}}(T)-\boldsymbol\theta\|\\
&
>\|\mathbf{u}_n(T)-\bar{\mathbf{u}}_n(T)-\boldsymbol\theta\|
-\|\mathbf{u}(T)-\mathbf{u}_n(T)\|-\|\bar{\mathbf{u}}_n(T)-\bar{\mathbf{u}}(T)\|\\
&
>\|\mathbf{u}_n(T)-\bar{\mathbf{u}}_n(T)-\boldsymbol\theta\|-2\epsilon,
\end{align*}
which yields
\[
{\color{black}\min_{\theta\in\Sset^1}}\|\mathbf{u}(T)-\bar{\mathbf{u}}(T)-\boldsymbol\theta\|>\epsilon.
\]
Thus, we conclude that the family $\U$  is unstable.
\end{proof}

\begin{rmk}\
\begin{enumerate}
\setlength{\leftskip}{-1.8em}
\item[\rm(i)]
In the proof of Theorem~{\rm\ref{thm:2d}(i)},
 if $\U_n$ converged to a stable family of solutions in the sense of $L^2(I)$
 as $n\to\infty$,
 Eq.~\eqref{eqn:thm2d} would not hold.
\item[\rm(ii)]
Only under the hypothesis of Theorem~{\rm\ref{thm:2c}},
 $\U$  is not necessarily unstable
  even if $\U_n$ is unstable for $n>0$ sufficiently large.
Moreover, $\U$ may be asymptotically stable
 even if $\U_n$ is unstable for $n>0$ sufficiently large.
This happens for the KM \eqref{eqn:dsys} and CL \eqref{eqn:csys}
 $($see Section~$7)$.
\end{enumerate}
\end{rmk}

{\color{black}
Without assuming the existence
 of the solution $\bar{\mathbf{u}}_n(t)$ to the KM \eqref{eqn:dsys}
 satisfying \eqref{eqn:2c} in Theorem~\ref{thm:2d}(ii),
 we obtain the following.

\begin{thm}
\label{thm:2e}
If $\U$ is unstable,
 then for any $\epsilon,\delta>0$ there exists $T>0$ such that for $n>0$ sufficiently large
\[
\min_{\theta\in\Sset^1}\|\mathbf{u}_n(T)-\bar{\mathbf{u}}(T)-\boldsymbol{\theta}\|>\epsilon,
\]
where $\mathbf{u}_n(t)$ is a solution to the KM \eqref{eqn:dsys} satisfying
\[
\min_{\theta\in\Sset^1}\|\mathbf{u}_n(0)-\bar{\mathbf{u}}(0)-\boldsymbol{\theta}\|<\delta.
\]
\end{thm}

\begin{proof}
Since $\U$ is unstable, for any $\delta,\epsilon>0$ there exists a solution $\mathbf{u}(t)$
such that
\[
\min_{\theta\in\Sset^1}\|\mathbf{u}(0)-\bar{\mathbf{u}}(0)-\boldsymbol{\theta}\|<\tfrac{1}{2}\delta,\quad
\min_{\theta\in\Sset^1}\|\mathbf{u}(T)-\bar{\mathbf{u}}(T)-\boldsymbol{\theta}\|>2\epsilon.
\]
for some $T>0$.
Let $\mathbf{u}(t)$ is a solution to the CL \eqref{eqn:csys} such that
\[
\|\mathbf{u}_n(0)-\mathbf{u}(0)\|=0.
\]
By Theorem~\ref{thm:2b}, for any $\epsilon>0$ there exists $N>0$ such that if $n>N$, then
\[
\|\mathbf{u}_n(T)-\mathbf{u}(T)\|<\epsilon.
\]
Hence, if $n>0$ is sufficiently large, then
\[
\|\mathbf{u}_n(0)-\mathbf{u}(0)\|<\tfrac{1}{2}\delta,
\]
so that
\[
\min_{\theta\in\Sset^1}\|\mathbf{u}_n(0)-\bar{\mathbf{u}}(0)-\boldsymbol{\theta}\|
 <\min_{\theta\in\Sset^1}\|\mathbf{u}(0)-\bar{\mathbf{u}}(0)-\boldsymbol{\theta}\|
 +\|\mathbf{u}_n(0)-\mathbf{u}(0)\|<\delta,
\]
and
\[
\min_{\theta\in\Sset^1}\|\mathbf{u}_n(T)-\bar{\mathbf{u}}(T)-\boldsymbol{\theta}\|
>\min_{\theta\in\Sset^1}\|\mathbf{u}(T)-\bar{\mathbf{u}}(T)-\boldsymbol{\theta}\|
 -\|\mathbf{u}_n(T)-\mathbf{u}(T)\|>\epsilon.
\]
This completes the proof.
\end{proof}

Theorem~\ref{thm:2e} says that
 if $\U$ is unstable, then it behaves as if it is an unstable family of solutions in the KM \eqref{eqn:dsys}.}


\section{Case $n=3$}

We now concentrate on the KM \eqref{eqn:dsys}
 and begin with the case of $n=3$:
\begin{equation}
\begin{split}
\dot{u}_1=&-\nu+\frac{K}{3}(\sin(u_2-u_1)+\sin(u_3-u_1)),\\
\dot{u}_2=&\frac{K}{3}(\sin(u_1-u_2)+\sin(u_3-u_2)),\\
\dot{u}_3=&\nu+\frac{K}{3}(\sin(u_1-u_3)+\sin(u_2-u_3)),
\end{split}
\label{eqn:dsys3}
\end{equation}
where $\nu=\tfrac{1}{3}a$.
For the KM \eqref{eqn:dsys} with $n=3$,
 a more general case in which the natural frequencies are not evenly spaced
 was discussed in \cite{MPBT04,MPT05a,MPT05b}.
Condition~\eqref{eqn:CD} becomes
\begin{equation}
C_\D=\frac{1}{3}+\frac{2}{3}\sqrt{1-\left(\frac{a}{3KC_\D}\right)^2}.
\label{eqn:CD3}
\end{equation}
As stated in Section~1,
 if there exists a constant $C_\D$ satisfying \eqref{eqn:CD3},
 then the system \eqref{eqn:dsys3} has the synchronized solutions
\begin{equation}
u_1=-\arcsin\left(\frac{a}{3KC_\D}\right)+\theta,\quad
u_2=\theta,\quad
u_3=\arcsin\left(\frac{a}{3KC_\D}\right)+\theta,
\label{eqn:dsol3}
\end{equation}
where $\theta\in\Sset^1$ is an arbitrary constant.
Moreover, $C_\D\ge\frac{1}{3}$ and 
\[
\frac{K}{a}\ge\kappa_0:=\frac{16\sqrt{2/(15+\sqrt{33})}}{4+\sqrt{34-2\sqrt{33}}}=0.56812\ldots
\]
since by \eqref{eqn:CD3}
\[
\frac{a}{K}=\psi_3\left(\frac{a}{3KC_\D}\right),\quad
\psi_3(\xi)=\xi+2\xi\sqrt{1-\xi^2},
\]
and $\psi_3(\xi)$ has a unique extremum (maximum) $\kappa_0^{-1}$ at
\[
\xi={
\sqrt{\frac{15+\sqrt{33}}{32}}}
\]
on $(0,1)$.
In particular, $K/a=1$ and $\arcsin(a/3KC_\D)=\tfrac{1}{2}\pi$ when $C_\D=\tfrac{1}{3}$, and
\begin{align*}
&
C_\D=C_{\D 0}:=\frac{1}{3}+\frac{1}{12}{(\sqrt{33}-1)}=0.72871\ldots,\\
&
\arcsin\left(\frac{a}{3KC_\D}\right)
 =\arcsin\sqrt{\frac{15+\sqrt{33}}{32}}=0.93592\ldots
\end{align*}
when $K/a=\kappa_0$.
See Fig.~\ref{fig:3a} for the dependence of $C_\D$ on $K/a$ in \eqref{eqn:CD3}.
  
Let $v_1=u_1-u_2$ and $v_2=u_3-u_2$.
We rewrite \eqref{eqn:dsys3} as
\begin{equation}
\begin{split}
\dot{v}_1=&-\nu-\frac{K}{3}(2\sin v_1+\sin v_2-\sin(v_2-v_1)),\\
\dot{v}_2=&\nu-\frac{K}{3}(\sin v_1+2\sin v_2+\sin(v_2-v_1)).
\end{split}
\label{eqn:dsys3r}
\end{equation}
The equilibria in \eqref{eqn:dsys3r} satisfy
\[
\sin v_1+\sin v_2=0,\quad
-\sin v_1+\sin v_2+2\sin(v_2-v_1)=\frac{6\nu}{K},
\]
i.e.,
\begin{equation}
v_1=-v_2,\quad
\sin v_2(1+2\cos v_2)=\frac{a}{K}
\label{eqn:dsol3a}
\end{equation}
or
\begin{equation}
v_1=v_2-\pi,\quad
v_2=\arcsin\left(\frac{a}{K}\right)\mbox{ or }\pi-\arcsin\left(\frac{a}{K}\right).
\label{eqn:dsol3b}
\end{equation}
The Jacobian matrix for the vector field of \eqref{eqn:dsys3r} at the equilibria
 is computed as
\[
A=-\frac{K}{3}\begin{pmatrix}
2\cos v_1+\cos(v_2-v_1) & {\cos v_2}-\cos(v_2-v_1)\\
\cos v_1-\cos(v_2-v_1) & 2\cos v_2+\cos(v_2-v_1)
\end{pmatrix}.
\]
Hence, $A$ only has eigenvalues with negative real parts
 and the equilibria are asymptotically stable if
\begin{equation}
\begin{split}
&
-\cos v_1-\cos v_2{-}\cos(v_2-v_1)<0\\
&
\mbox{and}\quad
\cos v_1\cos v_2+(\cos v_1+\cos v_2)\cos(v_2-v_1)>0,
\end{split}
\label{eqn:s3}
\end{equation}
and it has an eigenvalue with positive real parts and they are unstable if
\begin{equation}
\begin{split}
&
-\cos v_1-\cos v_2{-}\cos(v_2-v_1)>0\\
&
\mbox{or}\quad
\cos v_1\cos v_2+(\cos v_1+\cos v_2)\cos(v_2-v_1)<0
\end{split}
\label{eqn:u3}
\end{equation}
Solutions to \eqref{eqn:dsys3} corresponding to these equilibria in \eqref{eqn:dsys3r}
 are given by
\[
u_1=v_1+\theta,\quad
u_2=\theta,\quad
u_3=v_2+\theta,
\]
where $\theta\in\Sset^1$ is an arbitrary constant, since
\[
\dot{u}_2=\frac{K}{3}(\sin v_1+\sin v_2)=0.
\]

\begin{figure}
\includegraphics[scale=0.48]{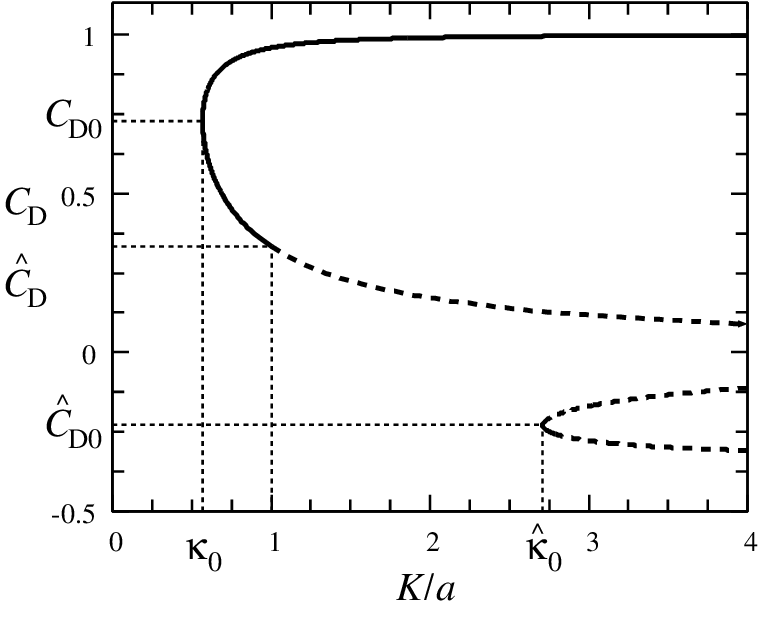}
\caption{Dependence of $C_\D$ and $\hat{C}_\D$ on $K/a$ in \eqref{eqn:CD3} and \eqref{eqn:hCD3}.
The solid and dashed lines represent it for $C_\D$ and $\hat{C}_\D$, respectively.
\label{fig:3a}}
\end{figure}

One of the equilibria satisfying \eqref{eqn:dsol3a} is given by
\begin{equation}
v_1=-\arcsin\left(\frac{a}{3KC_\D}\right),\quad
v_2=\arcsin\left(\frac{a}{3KC_\D}\right),
\label{eqn:dsol3r1}
\end{equation}
which corresponds to the synchronized solution \eqref{eqn:dsol3} in \eqref{eqn:dsys3},
 and the other is given by 
\begin{equation}
v_1=\arcsin\left(\frac{a}{3K\hat{C}_\D}\right)-\pi,\quad
v_2=-\arcsin\left(\frac{a}{3K\hat{C}_\D}\right)+\pi,
\label{eqn:dsol3r2}
\end{equation}
where
\begin{equation}
\hat{C}_\D=\frac{1}{3}-\frac{2}{3}\sqrt{1-\left(\frac{a}{3K\hat{C}_\D}\right)^2}.
\label{eqn:hCD3}
\end{equation}
We easily see that $\hat{C}_\D\le\tfrac{1}{3}$
 and that $K/a=1$ and  $v_2=-v_1=\pi/2$ when $\hat{C}_\D=\tfrac{1}{3}$.
Moreover, $\hat{C}_\D\to 0$ as $K/a\to\infty$, and
\[
\frac{a}{K}\ge\hat{\kappa}_0:=-\frac{16\sqrt{2/(15-\sqrt{33})}}{4-\sqrt{34+2\sqrt{33}}}
=2.70996\ldots
\]
since by \eqref{eqn:hCD3}
\[
\frac{a}{K}=\hat{\psi}_3\left(\left|\frac{a}{3K\hat{C}_\D}\right|\right),\quad
\hat{\psi}_3(\xi)=\xi-2\xi\sqrt{1-\xi^2}
\]
and $\hat{\psi}_3(\xi)$ has a unique extreme (maximum) $\hat{\kappa}_0^{-1}$ at
\[
\xi=\sqrt{\frac{15-\sqrt{33}}{32}}
\]
on $(0,1)$.
In particular,
\begin{align*}
&
\hat{C}_\D=\hat{C}_{\D 0}:=\frac{1}{3}-\frac{1}{12}\sqrt{34+2\sqrt{33}}=-0.22871\ldots,\\
&
\arcsin\left(\frac{a}{3K\hat{C}_\D}\right)
 =-\arcsin\sqrt{\frac{15-\sqrt{33}}{32}}=-0.56782\ldots
\end{align*}
when $K/a=\hat{\kappa}_0$.
See Fig.~\ref{fig:3a} for the dependence of $\hat{C}_\D$ on $K/a$ in \eqref{eqn:hCD3}.

On the other hand, the equilibria satisfying \eqref{eqn:dsol3b} exist when $K/a\ge 1$
 and coalesce at $K/a=1$.
Thus, a saddle-node bifurcation
 where the equilibria given by \eqref{eqn:dsol3r1} are born
 occurs at $K/a=\kappa_0$
 and a pitchfork bifurcation where the equilibria
 given by \eqref{eqn:dsol3b} and  \eqref{eqn:dsol3r2} are born occurs at $K/a=1$.
Moreover, the equilibria given by \eqref{eqn:dsol3r1}
 are asymptotically stable if $v_2=-v_1$ is less than
\[
v_0=\arcsin\sqrt{\frac{15+\sqrt{33}}{32}}=0.93592\ldots,
\]
and they are unstable if it is greater than  $v_0$,
 while the other equilibria are always unstable, as shown in Appendix~A.

\begin{figure}
\includegraphics[scale=0.58]{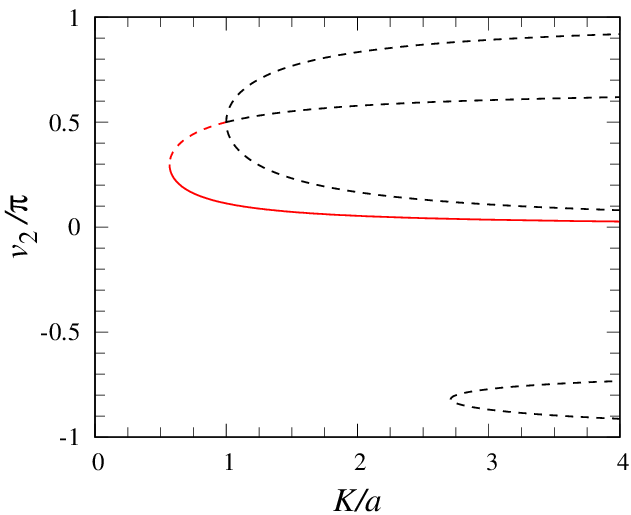}
\caption{Numerically computed bifurcation diagram of equilibria in \eqref{eqn:dsys3r}.
See the text for more details.
\label{fig:3b}}
\end{figure}

A numerically computed bifurcation diagram of equilibria in \eqref{eqn:dsys3r}
 is displayed in Fig.~\ref{fig:3b}.
Here the computer tool called \texttt{AUTO} \cite{DO12} was used.
The solid and dashed lines represent stable and unstable  equilibria, respectively.
The red and black lines, respectively, corresponds to the equilibrium given by \eqref{eqn:dsol3r1}
 and the others, which are given by \eqref{eqn:dsol3b} and \eqref{eqn:dsol3r2}. 
Only the equilibrium given by \eqref{eqn:dsol3r1} can be stable.


\section{Synchronized Solutions}

In this and the next two sections we consider the case
 in which the node number $n$ is odd and greater than three for \eqref{eqn:dsys},
 and discuss the existence of synchronized solutions and their bifurcations and stability.

Let $n=2n_0+1$ with $n_0>1$.
Equation~\eqref{eqn:dsys} becomes
\begin{equation}
\dot{u}_i=(i-n_0-1)\nu+\frac{K}{n}\sum_{j=1}^n\sin(u_j-u_i),\quad
i\in[n],
\label{eqn:dsyso}
\end{equation}
where $\nu=a/n$.
Condition~\eqref{eqn:CD} becomes
\begin{equation}
C_\D=\frac{1}{n}+\frac{2}{n}\sum_{j=1}^{n_0}\sqrt{1-\left(\frac{ja}{nKC_\D}\right)^2}.
\label{eqn:CDo}
\end{equation}
As stated in Section~1,
 if there exists a constant $C_\D$ satisfying \eqref{eqn:CDo},
 then the system \eqref{eqn:dsyso} has the synchronized solutions
\begin{equation}
u_i=\arcsin\left(\frac{(i-n_0-1)a}{nKC_\D}\right)+\theta,\quad
i\in[n],
\label{eqn:dsolo}
\end{equation}
where $\theta\in\Sset^1$ is an arbitrary constant.
Moreover,
\[
C_\D\ge C_{\D1}
 :=\frac{1}{n}+\frac{2}{n}\sum_{j=1}^{n_0}\sqrt{1-\left(\frac{j}{n_0}\right)^2},
\]
which is $\tfrac{1}{5}(1+\sqrt{3})=0.5464\ldots$ for $n=5$
 and $\tfrac{1}{55}(19+4\sqrt{6}+2\sqrt{21})=0.69023\ldots$ for $n=11$,
 and $K/a\ge\kappa_0\approx 0.60670\ldots$ for $n=5$
 and $0.62791\ldots$  for $n=11$, since by \eqref{eqn:CDo}
\[
\frac{a}{K}=\psi_n\left(\frac{n_0a}{nKC_\D}\right),\quad
\psi_n(\xi)=\frac{1}{n_0}\left(\xi+2\xi\sum_{j=1}^{n_0}\sqrt{1-\left(\frac{j}{n_0}\xi\right)^2}\right)
\]
and $\psi_n(\xi)$ has a unique extreme (maximum) $\kappa_0^{-1}$
 at $\xi=0.88209\ldots$ for $n=5$ and $0.94573\ldots$ for $n=11$.
We also have
\[
\frac{K}{a}=\kappa_1:=\frac{n_0}{n C_{\D 1}},
\]
which is $\sqrt{3}-1=0.73205\ldots$ for $n=5$
 and $25/(19+4\sqrt{6}+2\sqrt{21})= 0.65853\ldots$ for $n=11$ when $C_\D=C_{\D1}$,
 and $C_{\D 0}=0.74741\ldots$ for $n=5$ and $0.76543\ldots$ for $n=11$
 when $K/a=\kappa_0$.
See Fig.~\ref{fig:4a} for the dependence of $C_\D$ on $K/a$ in \eqref{eqn:CDo}
 for $n=5$ and $11$.
  
\begin{figure}
\includegraphics[scale=0.48]{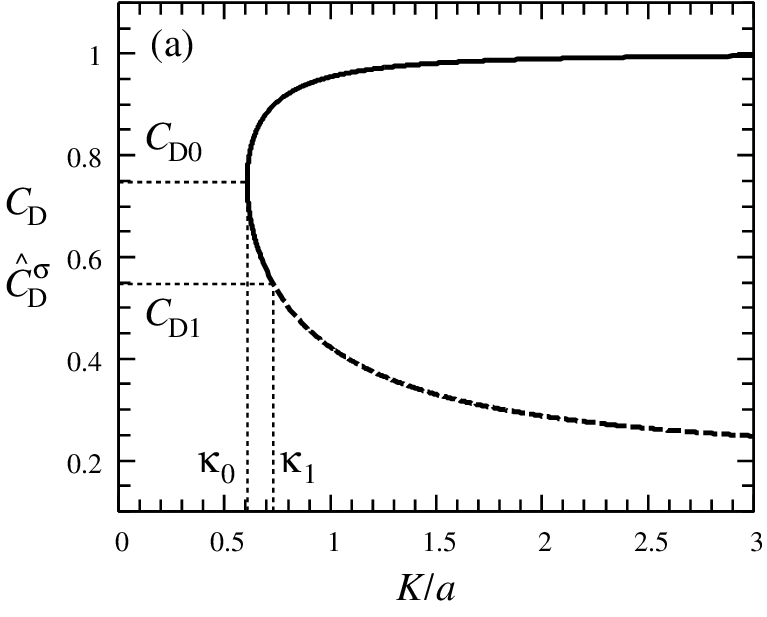}\
\includegraphics[scale=0.48]{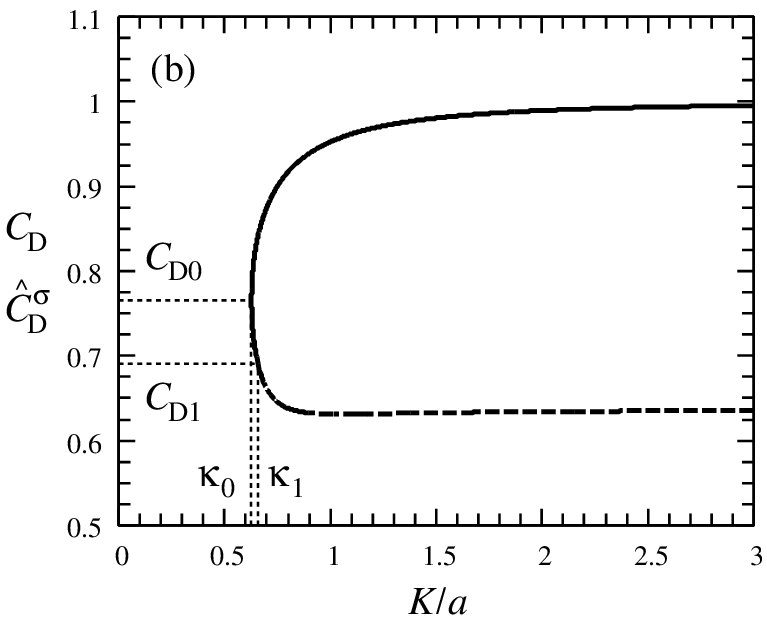}
\caption{Dependence of $C_\D$ and $\hat{C}_\D^\sigma$ on $K/a$ in \eqref{eqn:CDo}:
(a) $n=5$; (b) $n=11$.
The solid and dashed lines represent it for $C_\D$ and $\hat{C}_\D^\sigma$, respectively,
 where $\sigma=\{-1,1,1,-1\}$ and $\{-1,1,\ldots,1,-1\}$.
\label{fig:4a}}
\end{figure}

Let
\[
v_i=\begin{cases}
u_i-u_{n_0+1} &  \mbox{for $i\le n_0$};\\
u_{i+1}-u_{n_0+1} &  \mbox{for $n_0<i\le 2n_0$}.
\end{cases}
\]
We rewrite \eqref{eqn:dsyso} as
\begin{equation}
\begin{split}
\dot{v}_i
 =& (i-n_0-1)\nu-\frac{K}{n}\left(2\sin v_i+\sum_{j=1,j\neq i}^{2n_0}(\sin v_j-\sin(v_j-v_i))\right),\quad
 i\le n_0,\\
\dot{v}_i
 =& (i-n_0)\nu-\frac{K}{n}\left(2\sin v_i+\sum_{j=1,j\neq i}^{2n_0}(\sin v_j-\sin(v_j-v_i))\right),\quad
 n_0<i\le 2n_0,
\end{split}
\label{eqn:dsysor}
\end{equation}
where $\nu=a/n$.
We easily see that the system \eqref{eqn:dsysor} has an equilibrium
\begin{equation}
v_i=\begin{cases}
\displaystyle
\arcsin\left(\frac{(i-n_0-1)a}{nKC_\D}\right) &\mbox{for $i\le n_0$};\\[2.5ex]
\displaystyle
\arcsin\left(\frac{(i-n_0)a}{nKC_\D}\right) & \mbox{for $n_0<i\le 2n_0$},
\end{cases}
\label{eqn:dsolor0}
\end{equation}
which corresponds to the synchronized solution \eqref{eqn:dsolo} in \eqref{eqn:dsyso}
 and satisfies the relations
\begin{equation*}
v_i=-v_{2n_0-i+1}
\end{equation*}
and
\begin{equation*}
v_i\in[-\tfrac{1}{2}\pi,0),\quad
v_{n_0+i}\in(0,\tfrac{1}{2}\pi]
\end{equation*}
for $i\in[n_0]$.

Let $\sigma=\{\sigma_i\}_{i=1}^{2n_0}$ be a sequence of length $2n_0$
 with $\sigma_i\in\{-1,1\}$, $i\in[2n_0]$, and let
\[
\Sigma_{n_0}=\left\{\sigma=\{\sigma_i\}_{i=1}^{2n_0}\mid \sigma_i\in\{-1,1\},i\in[2n_0]\right\}.
\]
For each $\sigma\in\Sigma_{n_0}$,
 we define $\hat{C}_\D^\sigma$ such that it satisfies
\begin{align}
\hat{C}_\D^\sigma=\frac{1}{n}
&
+\frac{1}{n}\sum_{i=1}^{n_0}\sigma_i\sqrt{1-\left(\frac{(i-n_0-1)a}{nK\hat{C}_\D^\sigma}\right)^2}\notag\\
&
+\frac{1}{n}\sum_{i=n_0+1}^{2n_0}\sigma_i\sqrt{1-\left(\frac{(i-n_0)a}{nK\hat{C}_\D^\sigma}\right)^2},
\label{eqn:CDj}
\end{align}
and write
\begin{equation}
v_i^\sigma=\begin{cases}
\phi_i & \mbox{if $\sigma_i=1$;}\\
\pi-\phi_i & \mbox{if $\sigma_i=-1$ and $\phi_i>0$;}\\
-\phi_i-\pi & \mbox{if $\sigma_i=-1$ and $\phi_i<0$,}
\end{cases}
\label{eqn:dsolor}
\end{equation}
where
\begin{equation}
\phi_i=\begin{cases}
\displaystyle
\arcsin\left(\frac{(i-n_0-1)a}{nK\hat{C}_\D^\sigma}\right)
&
\mbox{for $i\le n_0$;}\\[3ex]
\displaystyle
\arcsin\left(\frac{(i-n_0)a}{nK\hat{C}_\D^\sigma}\right)
&
\mbox{for $n_0<i\le 2n_0$.}
\end{cases}
\label{eqn:phi}
\end{equation}
Note that $\phi_i\neq 0$, $i\in[2n_0]$.
We have the following.

\begin{thm}\
\label{thm:4a}
\begin{enumerate}
\setlength{\leftskip}{-1.8em}
\item[\rm(i)]
For each $\sigma\in\Sigma_{n_0}$,
 $v=v^\sigma\in\Tset^{2n_0}:=\prod_{i=1}^{2n_0}\Sset^1$
 gives an equilibrium in \eqref{eqn:dsysor}{,
 when $\hat{C}_\D^\sigma$ satisfies \eqref{eqn:CDj}.}
Moreover, no other equilibrium exists in \eqref{eqn:dsysor}.
\item[\rm(ii)]
Fix $K/a>0$ and $i\in[n_0]$.
If the equilibrium $v=v^\sigma$ with $\sigma_i=-1$ and $\sigma_{2n_0-i+1}=1$ exists,
 then so does $v=v^{\bar{\sigma}}$ with $\bar{\sigma}_i=1$, $\bar{\sigma}_{2n_0-i+1}=-1$
 and $\bar{\sigma}_j=\sigma_j$, $j\in[2n_0]\setminus\{i,2n_0-i+1\}$, and vice versa.
\end{enumerate}
\end{thm}

\begin{proof}
Let $w=(w_1,\ldots,w_{2n_0})$ denote an equilibrium of \eqref{eqn:dsysor}.
{
We begin with the following two lemmas.} 
\begin{lem}
\label{lem:4a}
We have
\[
{\sum_{i=1}^{2n_0}}\sin w_i=0.
\]
\end{lem}

\begin{proof}
Since the right-hand sides of \eqref{eqn:dsysor} are zero when $v_i=w_i$, $i\in[2n_0]$,
 we sum up them to obtain the desired result.
\end{proof}

Let
\[
C_w={\sum_{i=1}^{2n_0}}\cos w_i,
\]
and let $i'=2n_0-i+1$ for $i\in[n_0]$.

\begin{lem}
\label{lem:a2}
If $C_w\neq -1$, then
\[
\sin w_i+\sin w_{i'}=0.
\]
\end{lem}

\begin{proof}
Adding two of the right-hand sides of \eqref{eqn:dsysor} for $i$ and $i'$
 and using Lemma~\ref{lem:4a}, we have
\[
(C_w+1)(\sin w_i+\sin w_{i'})=0.
\]
This yields the desired result.
\end{proof}

Using Lemma~\ref{lem:a2} and \eqref{eqn:dsysor}, we have
\begin{align*}
\begin{array}{ll}
\displaystyle
(i-n_0-1)\nu-\frac{K}{n}(C_w+1)\sin w_i=0
&
\mbox{for $i\le n_0$;}\\
\displaystyle
(i-n_0)\nu-\frac{K}{n}(C_w+1)\sin w_i=0
&
\mbox{for $n_0<i\le 2n_0$.}
\end{array}
\end{align*}
Obviously, $C_w\neq -1$ since the above relations do not hold otherwise.
Hence, $w_i=\phi_i$ or $w_i=\pi-\phi_i$, $ i\in[2n_0]$,
 where $\phi_i$ is given by \eqref{eqn:phi} with $\hat{C}_\D^\sigma=(C_w+1)/n$,
 since $\nu=a/n$.
This yields part~(i).
{
By \eqref{eqn:CDj}}
 we have $\hat{C}_\D^\sigma=\hat{C}_\D^{\bar{\sigma}}$, 
 from which part~(ii) follows. 
Thus, we complete the proof.
\end{proof}

In particular, $v=v^\sigma$ is the same as the equilibrium given by \eqref{eqn:dsolor0}
 and $\hat{C}_\D^\sigma=C_\D$ when $\sigma=\{1,\ldots,1\}$.
Moreover, $\hat{C}_\D^\sigma=\hat{C}_\D$ when $\sigma=\{-1,-1\}$
 in the case of $n=3$.
See Fig.~\ref{fig:4a} for the dependence of $\hat{C}_\D^\sigma$ on $K/a$
 when $\sigma=\{-1,1,1,-1\}$ and $\{-1,1,\ldots,1,-1\}$ for $n=5$ and $n=11$, respectively. 
Solutions to \eqref{eqn:dsyso} corresponding to these equilibria in \eqref{eqn:dsysor}
 are given by
\begin{equation}
u_{n_0+1}=\theta,\quad
u_i=
\begin{cases}
v_i+\theta & \mbox{for $i\le n_0$};\\
v_{i-1}+\theta  & \mbox{for $i\ge n_0+2$},
\end{cases}
\label{eqn:dsyssol}
\end{equation}
where $\theta\in\Sset^1$ is an arbitrary constant, since
\[
\dot{u}_{n_0+1}=\frac{K}{n}\sum_{i=1}^{2n_0}\sin{v_i}=0
\]
{\color{black}at the equilibria} by Lemma~\ref{lem:4a}.

For $\sigma\in\Sigma_{n_0}$, we define 
\[
\chi^\sigma(\xi)
 =\frac{\xi}{n_0}\left(1+\sum_{i=1}^{n_0}\sigma_i\sqrt{1-\left(\frac{i-n_0-1}{n_0}\xi\right)^2}
 +\sum_{i=n_0+1}^{2n_0}\sigma_i\sqrt{1-\left(\frac{i-n_0}{n_0}\xi\right)^2}\right),
\]
where $\xi\in[0,1]$.
Obviously, $\chi^\sigma(0)=0$ for any $\sigma\in\Sigma_{n_0}$.
In general, we can choose an integer $\ell\in[n_0]\cup\{0\}$
 and a monotonically increasing sequence of integers $\{i_k\}_{k=1}^{\ell}$
 with $i_k\in[n_0]$, $k\in[\ell]$, such that
 $\sigma_j=\sigma_{2n_0-j+1}$ if $j=i_k$
 and $\sigma_j\neq\sigma_{2n_0-j+1}$ if $j\neq i_k$ for any $k\in[\ell]${\color{black},
 where  if $\ell=0$, then $\sigma_j\neq\sigma_{2n_0-j+1}$ for any $j\in[n_0]$.}
So we write
\begin{align}
\chi^\sigma(\xi)
=&
\frac{\xi}{n_0}\left(1+\sum_{k=1}^{\ell}(\sigma_{i_k}+\sigma_{2n_0-i_k+1})
 \sqrt{1-\left(\frac{i_k-n_0-1}{n_0}\xi\right)^2}\right)\notag\\
=&\frac{\xi}{n_0}\left(1+2\sum_{k=1}^\ell\sigma_{i_k}\sqrt{1-\left(\frac{i_k-n_0-1}{n_0}\xi\right)^2}\right).
\label{eqn:chi}
\end{align}
Letting $\xi=n_0a/nK|\hat{C}_\D^\sigma|$, we rewrite the relation \eqref{eqn:CDj} as
\begin{equation}
\frac{a}{K}=\left|\chi^\sigma(\xi)\right|.
\label{eqn:con}
\end{equation}
Hence, we have $\xi>0$ for any $K>0$ since $\chi^\sigma(0)=0$.
In particular, $\psi_n(\xi)=\chi^\sigma(\xi)$ when $\sigma_j=1$, $j\in[2n_0]$.
From Theorem~\ref{thm:4a}(i)
 we immediately obtain the following corollary.
 
\begin{cor}
\label{cor:4a}
If $\xi\in(0,1]$ satisfies \eqref{eqn:con} for $\sigma\in\Sigma_{n_0}$ and $K>0$,
 then $v^\sigma$ given by \eqref{eqn:dsolor} with 
\begin{equation}
\phi_i=\begin{cases}
\displaystyle
{\pm}\arcsin\left(\frac{(i-n_0-1)\xi}{n_0}\right)
&
\mbox{for $i\le n_0;$}\\[2ex]
\displaystyle
{\pm}\arcsin\left(\frac{(i-n_0)\xi}{n_0}\right)
&
\mbox{for $n_0<i\le 2n_0,$}
\end{cases}
\label{eqn:cor4a}
\end{equation}
instead of \eqref{eqn:phi} is an equilibrium in \eqref{eqn:dsysor},
 {where the upper or lower sign is taken,
  depending on whether $\chi^\sigma(\xi)$ is positive or not.}
\end{cor}

\begin{rmk}\
\label{rmk:4a}
\begin{enumerate}
\setlength{\leftskip}{-1.8em}
\item[(i)]
It is clear that
 $v^\sigma\neq v^{\hat{\sigma}}$
 if $\sigma\neq\hat{\sigma}$ and $\xi\neq 1$.
\item[(ii)]
{Mirollo and Strogatz {\rm\cite{MS05}},
 and Verwoerd} and Mason {\rm\cite{VM08}} used
 an expression similar to the function $\chi^\sigma(\xi)$
 $($and equivalently to Eq.~\eqref{eqn:CDj}$)$
 to obtain equilibria in the KM \eqref{eqn:dsys}
 when $\omega_i$, $i\in[n]$, are not necessarily evenly spaced.
In particular, the former half of Theorem~{\rm\ref{thm:4a}(i)}
 is equivalent to a special case of Theorem~$2$ of {\rm\cite{VM08}}
 although it contains stronger statements.
\end{enumerate}
\end{rmk}


\section{Bifurcations}

We next discuss bifurcations of the equilibria in \eqref{eqn:dsysor}
 detected by Theorem~\ref{thm:4a} and equivalently by Corollary~\ref{cor:4a}.
The relation \eqref{eqn:con} gives a branch of equilibria when $K$ is taken as a control parameter.
So we have the following result.

\begin{thm}\
\label{thm:5a}
\begin{enumerate}
\setlength{\leftskip}{-1.8em}
\item[\rm(i)]
The equilibrium $v^\sigma$ suffers
 a supercritical $($resp. subcritical$)$ saddle-node bifurcation at
\begin{equation}
K=\frac{a}{\left|\chi^{\sigma}(\xi_0)\right|}
\label{eqn:thm5a0}
\end{equation}
in \eqref{eqn:dsysor} with $\xi=\xi_0$ on $(0,1)$
 if and only if $|\chi^\sigma(\xi)|$ has a local maximum $($resp. minimum$)$.
In particular, if
\begin{equation}
\sigma_1,\sigma_{2n_0}=1,\quad
\chi^\sigma(1)\ge 0
\label{eqn:thm5a1}
\end{equation}
or
\begin{equation}
\sigma_1,\sigma_{2n_0}=-1,\quad
\chi^\sigma(1)\le 0,
\label{eqn:thm5a2}
\end{equation}
then a supercritical saddle-node bifurcation occurs.
Moreover, if $\sigma_i=1$, $i\in[2n_0]$,
 then {$\chi^\sigma(\xi)$ has a unique local maximum and no local minimum, and
 $v^\sigma$ suffers only one saddle-node bifurcation.
\item[\rm(ii)]
Let $v^{\sigma^{\pm\pm}}$ be four equilibria in \eqref{eqn:dsysor} such that
\[
\sigma_i^{++}=\sigma_i^{--}=\sigma_i^{+-}=\sigma_i^{-+},\quad
i\neq 1,2n_0,
\]
and
\[
\sigma_1^{++},\sigma_1^{+-}=1,\quad
\sigma_1^{-+},\sigma_1^{--}=-1,\quad
\sigma_{2n_0}^{++},\sigma_{2n_0}^{-+}=1,\quad
\sigma_{2n_0}^{+-},\sigma_{2n_0}^{--}=-1.
\]
If $\chi^\sigma(1)\neq 0$, then a pitchfork bifurcation where
 $v^{\sigma^{++}}$ changes to $v^{\sigma^{--}}$
 and where $v^{\sigma^{+-}}$ and $v^{\sigma^{-+}}$ are born occurs at
\begin{equation}
K=\frac{a}{\left|\chi^{\sigma}(1)\right|},
\label{eqn:thm5a3}
\end{equation}
where any of $\sigma^{\pm\pm}$ may be chosen as $\sigma$.
Moreover, the bifurcation is super- or subcritical, depending on whether
\begin{equation}
\chi^\sigma(1)\frac{\d\chi^{\sigma}}{\d\xi}(1)\quad
\mbox{with $\sigma=\sigma^{+-}$ and $\sigma^{-+}$}
\label{eqn:thm5a4}
\end{equation}
is positive or negative, where any of $\sigma^{+-}$ and $\sigma^{-+}$
 may be chosen as $\sigma$.}
\end{enumerate}
\end{thm}

\begin{proof}
We see via Corollary~\ref{cor:4a} that if $\xi\in(0,1)$ satisfies \eqref{eqn:con},
 then there exists the equilibrium $v^\sigma$ with $|\hat{C}_\D^\sigma|=n_0a/nK\xi$.
Hence, If $|\chi^\sigma(\xi)|$ has a {local} maximum (resp. minimum)
 at $\xi=\xi_0$ on $(0,1)$,
 then {near \eqref{eqn:thm5a0},
 no value of $\xi$ satisfies \eqref{eqn:con} for smaller (resp. larger) values of $K$ than it
 but two values of $\xi$ satisfies it for larger (resp. smaller) values of $K$ than it,}
 so that the equilibrium $v^\sigma$ suffers
 a supercritical $($resp. subcritical$)$ saddle-node bifurcation there.

If condition~\eqref{eqn:thm5a1} or \eqref{eqn:thm5a2} holds,
 then $|\chi^\sigma(\xi)|$ has a {local} maximum,
 since $\chi^\sigma(0)=0$ and
\begin{equation}
\frac{\d \chi^{\sigma}}{\d\xi}(\xi)\to -\infty
\quad\mbox{or}\quad
\frac{\d \chi^{\sigma}}{\d\xi}(\xi)\to +\infty
\label{eqn:thm5a5}
\end{equation}
as $\xi\to 1$, depending on whether $\sigma_1,\sigma_{2n_0}=1$ or $-1$.
Moreover, when $\sigma_i=1$, $i\in[2n_0]$,
\[
\chi^\sigma(\xi)
=\frac{\xi}{n_0}\left(1+2\sum_{i=1}^{n_0}\sqrt{1-\left(\frac{i}{n_0}\xi\right)^2}\right)>0
\]
and
\begin{align*}
\frac{\d\chi^\sigma}{\d\xi}(\xi)
=&\frac{1}{n_0}\left(1+2\sum_{i=1}^{n_0}\biggl(1-2\left(\frac{i}{n_0}\xi\right)^2\biggr)
 \Bigg/\sqrt{1-\left(\frac{i}{n_0}\xi\right)^2}\right)\\
{
=}&
{
\frac{1}{n_0}\left(1+2\sum_{i=1}^{n_0}
\left(2\sqrt{1-\left(\frac{i}{n_0}\xi\right)^2}-1\Bigg/\sqrt{1-\left(\frac{i}{n_0}\xi\right)^2}\right)\right)}
\end{align*}
is monotonically decreasing on $(0,1)$ and goes to $-\infty$ as $\xi\to 1$,
 so that $\chi^\sigma(\xi)$ has a unique local maximum and no local minimum.
Thus, we prove part~(i).

We turn to the proof of part~(ii).
When $\xi=1$,
 the four equilibria coincide
 and the corresponding functions $\chi^{\sigma^{\pm\pm}}(\xi)$ have the same value.
To obtain the desired result, we only have to notice that
 $\sigma^{++}$ and $\sigma^{--}$, respectively,
 satisfy the first and second equations of \eqref{eqn:thm5a5} as $\xi\to 1$,
 and that $v^{\sigma^{+-}}$ and $v^{\sigma^{-+}}$ exist in a pair by Theorem~\ref{thm:4a}(ii)
 for larger or smaller  values of $K$ than and near \eqref{eqn:thm5a3},
 depending on whether the quantity \eqref{eqn:thm5a4},
 which takes the same value for both of $\sigma=\sigma^{+-}$ and $\sigma^{-+}$,
 is positive or negative.
\end{proof}

\begin{rmk}\
\label{rmk:5a}
\begin{enumerate}
\setlength{\leftskip}{-1.6em}
\item[\rm(i)]
It follows from Theorem~{\rm\ref{thm:5a}(i)} that
 one of the equilibria $v^\sigma$ with $\sigma=\sigma^{++}$ or $\sigma^{--}$
 in Theorem~{\rm\ref{thm:5a}(ii)} suffers a saddle-node bifurcation.
Moreover, the equilibrium $v^\sigma$ with $\sigma=\sigma^{+-}$ and $\sigma^{-+}$
 may also suffer a saddle-node bifurcation.
Indeed, when $\sigma_i\neq\sigma_{2n_0-i+1}$, $i\in[n_0]\setminus\{2\}$
 and $\sigma_2=\sigma_{2n_0-1}=1$, we have
\[
\chi^\sigma(\xi)=\frac{\xi}{n_0}\left(1+2\sqrt{1-\left(\frac{n_0-1}{n_0}\xi\right)^2}\right)>0
\]
on $(0,1]$ $($see \eqref{eqn:chi}$)$ and
\[
\frac{\d\chi^\sigma}{\d\xi}(\xi)
=\frac{1}{n_0}\left(1+2\biggl(1-2\left(\frac{n_0-1}{n_0}\xi\right)^2\biggr)
 \Bigg/\sqrt{1-\left(\frac{n_0-1}{n_0}\xi\right)^2}\right),
\]
which is monotonically decreasing.
If ${n_0\ge 6}$, then $\chi^\sigma(\xi)$ has a local maximum
 since $(\d\chi^\sigma/\d\xi)(1)<0$.
This means the claim by Theorem~{\rm\ref{thm:5a}(i)}.
\item[\rm(ii)]
In Theorem~{\rm\ref{thm:5a}(ii)},
if $(\d^j\chi^\sigma/\d\xi^j)(1)=0$, $j\in[\ell]$, for some $\ell\in\Nset$, then Eq.~\eqref{eqn:thm5a4} is replaced by
\[
\chi^\sigma(1)\frac{\d^{\ell+1}\chi^{\sigma}}{\d\xi^{\ell+1}}(1)\quad
\mbox{with $\sigma=\sigma^{+-}$ and $\sigma^{-+}$}.
\]
\item[\rm(iii)]
From Theorems~$\ref{thm:4a}$ and $\ref{thm:5a}$
 we see that in the reduced system \eqref{eqn:dsysor}
 $($resp. in the KM \eqref{eqn:dsyso}$)$,
 $2^{2n_0}$ or $3\cdot 2^{2(n_0-1)}$ equilibria $($resp. families of equilibria$)$ are born,
 depending on whether ones with  $\sigma=\sigma^{++}$ and $\sigma^{-+}$
 in Theorem~{\rm\ref{thm:5a}(ii)}
 are distinguished or not,
 and that $2^{2(n_0-1)}$ saddle-node bifurcations occur
 at least since a supercritical one occurs if Eq.~\eqref{eqn:thm5a1} or \eqref{eqn:thm5a2} holds.
\end{enumerate}
\end{rmk}

Moreover, we have the following on the number of pitchfork bifurcations
 occurring in these systems.

\begin{prop}
\label{prop:5a}
In the reduced system \eqref{eqn:dsysor} and equivalently in the KM \eqref{eqn:dsyso},
 the following hold on the number of pitchfork bifurcations$:$
\begin{enumerate}
\setlength{\leftskip}{-1.8em}
\item[(i)]
{
$2^{2n_0-3}+2^{n_0-2}n_0$} pitchfork bifurcations occur at least$;$
\item[(ii)]
$2^{2(n_0-1)}$ pitchfork bifurcations occur {at least} if $n_0$ is prime.
\end{enumerate}
\end{prop}

\begin{proof}
We begin with the proof of part~(i).
By Theorem~\ref{thm:5a}(ii) and Remark~\ref{rmk:5a}(ii),
 we only have to estimate the number of $\sigma\in\Sigma_{n_0}$
 such that  $\chi^\sigma(1)\neq{ 0}$.
Let $4\bar{n}$ denote the number.
{\color{black}If $\sigma_j\neq\sigma_{2n_0-j+1}$ for any $j\in[n_0]\setminus\{1\}$,
 then $\chi^\sigma(1)=1/n_0$.}
If for some $\ell\in[n_0-1]$ there exists a sequence $\{i_k\}_{k=1}^\ell$ 
 such that $i_k\in[n_0]\setminus\{1\}$, $\sigma_j=\sigma_{2n_0-j+1}$ for $j=i_k$
 and $\sigma_j\neq\sigma_{2n_0-j+1}$ for $j\neq i_k$, then
\begin{equation}
\chi^\sigma(1)
=\frac{1}{n_0}\left(1+2\sum_{k=1}^\ell\sigma_{i_k}\sqrt{1-\left(\frac{i_k-n_0-1}{n_0}\right)^2}\right).
\label{eqn:prop5a1}
\end{equation}
Obviously, 
 if $\ell=1$, then $\chi^\sigma(1)\neq 0$.
Assume that $\ell\ge 2$
 and let $\sigma^{\pm}\in\Sigma_{n_0}$ be such sequences
 satisfying $\sigma_{i_\ell}^\pm=\pm 1$
 and $\sigma_{i_j}^+=\sigma_{i_j}^-$ for $j\in[\ell-1]$.
If $\chi^{\sigma^+}(1)=0$, then $\chi^{\sigma^-}(1)\neq 0$, and vice versa.
The number of how to choose the pair $\sigma^\pm$ is $2^{\ell-1}$
 for each sequence $\{i_j\}_{j=1}^\ell$.
{\color{black}Since there are four sequences having the same subsequence $\{\sigma_j\}_{j=2}^{2n_0-1}$,
 we have}
\begin{align*}
\bar{n}\ge&
2^{n_0-1}+2^{n_0-2}\cdot 2(n_0-1)
 +\sum_{\ell=2}^{n_0-1}\frac{(n_0-1)!}{(n_0-\ell-1)!\ell!}2^{\ell-1}2^{n_0-\ell-1}\\
=&2^{2n_0-3}+2^{n_0-2}n_0,
\end{align*}
which yields part~(i).

We turn to the proof of part~(ii).
Again, by Theorem~\ref{thm:5a}(ii) and Remark~\ref{rmk:5a}(ii),
 we only have to show that $\chi^\sigma(1)\neq{0}$ for any $\sigma\in\Sigma_{n_0}$
 when $n_0$ is prime.
We  prove it by contradiction
 and assume that $\chi^\sigma(1)$ is given by \eqref{eqn:prop5a1}
 and  $\chi^\sigma(1)=0$ for some $\sigma\in\Sigma_{n_0}$ as well as $n_0$ is prime.
For some $i_k\neq 1$ and relatively prime positive integers $q,p$ with $p\ge 2$,
 we must have
\[
\sqrt{1-\left(\frac{i_k-n_0-1}{n_0}\right)^2}=\frac{q}{p}\in(0,1),
\]
i.e.,
\begin{equation}
n_0^2-(i_k-n_0-1)^2=\frac{q^2n_0^2}{p^2},\quad
p>q>0,
\label{eqn:prop5a2}
\end{equation}
since square roots of non-square integers are linearly independent over $\Zset$
 (see, e.g., \cite{B40}).
Moreover, 
 if $(p,q)=(p_{j_k},q_{j_k})$, $k\in[m]$, satisfy \eqref{eqn:prop5a2} for some $m\in[\ell]$,
 then we find a sequence $\{s_j\}_{j=1}^m$ such that $s_j\in\{0,\pm 1\}$, $k\in[m]$, and
\[
\sum_{k=1}^{m}s_k\frac{p_{j_k}}{q_{j_k}}=-\tfrac{1}{2}.
\]

Since $n_0$ is prime, it follows from \eqref{eqn:prop5a2} that $n_0=p$.
So we have
\begin{equation}
q^2+(p-i_k+1)^2=p^2.
\label{eqn:prop5a3}
\end{equation}
Hence, $(q,p-i_k+1,p)$ is a \emph{Pythagorean triple} \cite{M07}.
We need the following elementary result on Pythagorean triples.

\begin{lem}
\label{lem:5a}
Let $\ell_j\in\Nset$, $j=1,2,3$, and $\ell_1,\ell_3$ are relatively prime.
The triple $(\ell_1,\ell_2,\ell_3)$ is Pythagorean, i.e., $\ell_1^2+\ell_2^2=\ell_3^2$,
 if and only if there exist relatively prime integers $m_1,m_2$ with $m_1>m_2>0$
 such that $m_1,m_2$ are of opposite parity, i.e., one is even and the other is odd,
\[
(\ell_1,\ell_2)=(m_1^2-m_2^2,2m_1m_2)\mbox{ or }(2m_1m_2,m_1^2-m_2^2)
\] 
and $\ell_3=m_1^2+m_2^2$.
\end{lem}

See, e.g., Section~13.2 of \cite{HW79} or Chapter~1 and Appendix~B of \cite{M07}
 for a proof of Lemma~\ref{lem:5a}.
Note that if the triple $(\ell_1,\ell_2,\ell_3)$ is Pythagorean
 and $\ell_1,\ell_3$ are relatively prime,
 then so are $\ell_1,\ell_2,\ell_3$.

We also need the following elementary result on a sum of two squares
 (see, e.g., Section~15.1 of \cite{HW79} for a proof).

\begin{lem}
\label{lem:5b}
Let $\ell_3>2$ be a prime integer.
There exist $m_1,m_2\in\Nset$ such that $\ell_3=m_1^2+m_2^2$ and $m_1>m_2$
 if and only if $\ell_3=1\mod 4$.
Moreover, the pair $(m_1,m_2)$ is unique.
\end{lem}

We return to the proof of part~(ii) {in Proposition~\ref{prop:5a}}.
It follows from Lemmas~\ref{lem:5a} and \ref{lem:5b} that
 if Eq.~\eqref{eqn:prop5a3} holds, then there uniquely exists a pair $(m_1,m_2)$ of integers 
 such that $m_1>m_2>0$ and
\[
p=m_1^2+m_2^2,\quad
q=2m_1m_2\mbox{ or }m_1^2-m_2^2.
\]
{\color{black}Since $p=n_0$, we have $m=2$ and}
\[
s_1\frac{2m_1m_2}{m_1^2+m_2^2}+s_2\frac{m_1^2-m_2^2}{m_1^2+m_2^2}=-\tfrac{1}{2}
\]
for some $s_j\in\{0,\pm 1\}$, but this never occurs.
So we have a contradiction.
This completes the proof {of Proposition~\ref{prop:5a}}.
\end{proof}

We leave our discussion on saddle-node and pitchfork bifurcations in \eqref{eqn:dsysor}.

\begin{prop}
\label{prop:5b}
For any $\sigma\in\Sigma_{n_0}$,
 the equilibrium $v^\sigma$ suffers no Hopf bifurcation {in \eqref{eqn:dsysor}}.
\end{prop}
 
\begin{proof}
We first notice that
 the one-parameter family of equilibria given by \eqref{eqn:dsyssol} in \eqref{eqn:dsyso}
 has the same stability type as the corresponding equilibrium in \eqref{eqn:dsysor}.
Moreover, the former exhibits Hopf bifurcations
 if and only if so does the latter. 
So we analyze the KM \eqref{eqn:dsyso} instead of \eqref{eqn:dsysor} in the following.

We compute each element of the Jacobian matrix $A$
 for the vector field of \eqref{eqn:dsyso} as
\begin{equation}
A_{ij}=\begin{cases}
\displaystyle
-\frac{K}{n}\biggl(\cos v_i+\sum_{j=1,j\neq i}^{2n_0}\cos(v_j-v_i)\biggr) & \mbox{if $i=j$};\\
\displaystyle
\frac{K}{n}\cos(v_j-v_i) & \mbox{if $i\neq j\neq n_0+1$};\\[2ex]
\displaystyle
\frac{K}{n}\cos v_i & \mbox{if $i\neq j=n_0+1$}
\end{cases}
\label{eqn:A1}
\end{equation}
for $i\neq n_0+1$ and
\begin{equation}
A_{n_0+1,j}=\begin{cases}
\displaystyle
-\frac{K}{n}\sum_{j=1}^{2n_0}\cos v_j & \mbox{if $j=n_0+1$};\\
\displaystyle
\frac{K}{n}\cos v_j & \mbox{if $j\neq n_0+1$}.
\end{cases}
\label{eqn:A2}
\end{equation}
Thus, the matrix $A$ is symmetric and consequently only has real eigenvalues.
This implies the desired result.
\end{proof}

\begin{rmk}
\label{rmk:5b}
From Remark~{\rm\ref{rmk:4a}(i)} and Proposition~$\ref{prop:5b}$ we see that
 no other bifurcation than ones detected in Theorem~$\ref{thm:5a}$ occurs
 for equilibrium in the system \eqref{eqn:dsyso}.
\end{rmk}

\begin{figure}
\includegraphics[scale=0.58]{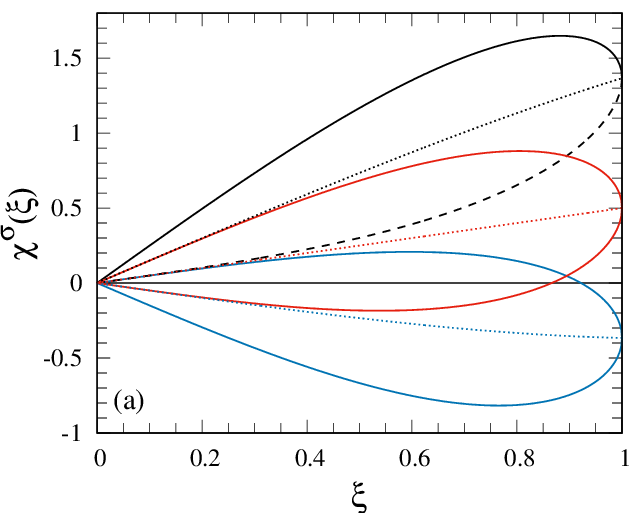}\\[1ex]
\includegraphics[scale=0.58]{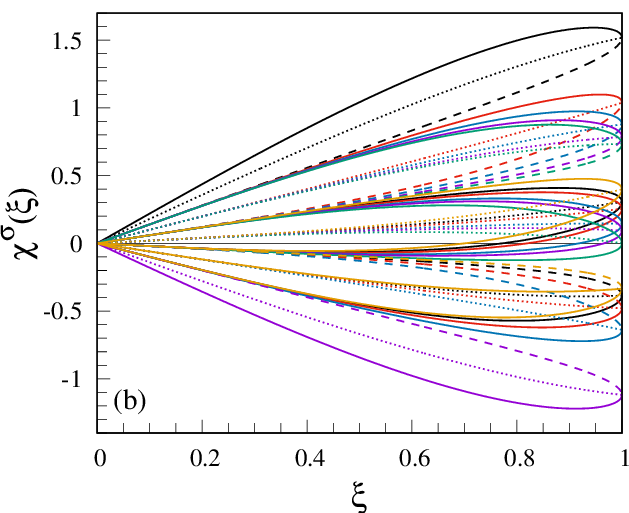}\
\includegraphics[scale=0.58]{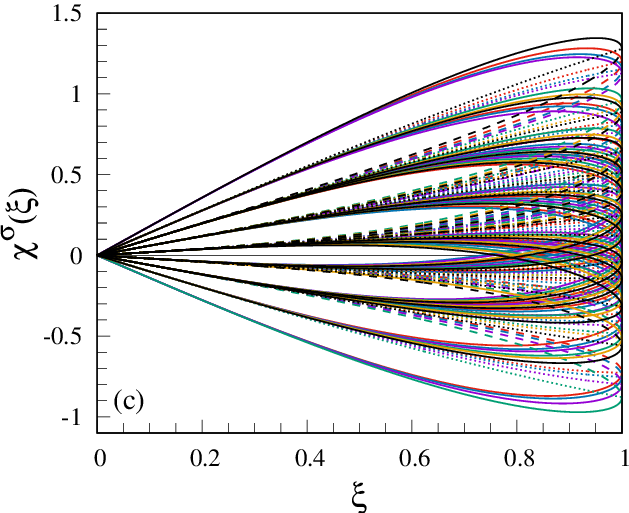}
\caption{Function $\chi^\sigma(\xi)$:
(a) $n=5$; (b) and (c) $n=11$.
See the text for more details.
\label{fig:5a}}
\end{figure}

The graph of $\chi^\sigma(\xi)$ is displayed
 for each $\sigma\in\Sigma_{n_0}$ when $n=5$ and $11$ in Fig.~\ref{fig:5a}. 
It is plotted as a solid line if $\chi^\sigma(\xi)$ has a local maximum or minimum,
 and as a dotted or dashed line otherwise,
 depending on whether $\sigma_1\neq\sigma_{2n_0}$ or not.
In Fig.~\ref{fig:5a}(a),
 the black and blue lines represent the graphs
 for $(\sigma_2,\sigma_3)=(1,1)$ and  $(-1,-1)$, respectively,
 while the red line for $(\sigma_2,\sigma_3)=(1,-1)$ or $(-1,1)$.
In Figs.~\ref{fig:5a}(b) and (c), respectively,
 the graphs are plotted for $\sigma\in\Sigma_{n_0}$
 with $\sigma_i=\sigma_{2n_0-i+1}$ for any $i\neq 1$
 and with $\sigma_i\neq\sigma_{2n_0-i+1}$ for some $i\neq 1,2n_0$
 when $n=11$.
See Appendix~B for the meaning of the line colors in Figs.~\ref{fig:5a}(b) and (c).


\section{Stability}

For the KM \eqref{eqn:dsyso} and reduced system \eqref{eqn:dsysor}
 we finally discuss the stability of equilibria
 detected by Theorem~\ref{thm:4a} and equivalently by Corollary~\ref{cor:4a}.
Unlike \eqref{eqn:dsys3r},
 it seems difficult to determine the stability of the equilibria in \eqref{eqn:dsysor}.
However,  we can prove the following theorem.

\begin{thm}\
\label{thm:6a}
\begin{enumerate}
\setlength{\leftskip}{-1.8em}
\item[(i)]
The equilibrium $v^\sigma$ with $\sigma_i=1$ for all $i\in[2n_0]$
 is asymptotically stable if $\xi=n_0a/nK|\hat{C}_\D^\sigma|<\xi_0$
 and unstable if $\xi>\xi_0$, where $\chi^\sigma(\xi_0)$ is the unique local maximum 
 detected in Theorem~{\rm\ref{thm:5a}(i)}.
\item[(ii)]
The equilibrium $v^\sigma$ is unstable
 if $\sigma_i=-1$ for some $i\in[2n_0]$.
\end{enumerate}
\end{thm}

\begin{proof}
Fix $\sigma\in\Sigma_{n_0}$
 and increase $\xi\in(0,1)$ along a branch of equilibria,
 $v^\sigma$ with $\chi^\sigma(\xi)$.
{\color{black}
Let $A$ denote the Jacobian matrix for the vector field of \eqref{eqn:dsyso},
 as in the proof of Proposition~\ref{prop:5b}.}
We have  the following lemma.

\begin{lem}
\label{lem:6a}
Suppose that the multiplicity of the zero eigenvalue of $A$
 for $v^\sigma$ with $\chi^\sigma(\xi)$, $\sigma\in\Sigma_{n_0}$,
 increases or decreases  at $\xi=\xi_*$ with $\chi^\sigma(\xi_\ast)\neq 0$,
 when the value of $\xi\in(0,1)$ is increased.
Then $\chi^\sigma(\xi)$ has an extremum at $\xi=\xi_*$
 and the multiplicity changes by one.
\end{lem}

\begin{proof}
{\color{black}
Let $\sigma\in\Sigma_{n_0}$ be fixed and let $\beta=a/K$.
We write $v^\sigma(\xi)=v^\sigma\in\Tset^{2n_0}$ for $\xi\in(0,1)$
 and $Kf(v;\beta)$ for the the vector field of \eqref{eqn:dsysor},
 i.e., it becomes $f(v;\beta)$ when the time is changed as $t\to t/K$.
Note that $f(v;\beta)$ is independent of $K$.
The Jacobian matrix $\D_v f(v^\sigma(\xi);\beta)$ is symmetric like that of the KM \eqref{eqn:dsys},
 so that its eigenvalues are real and geometrically simple.
In particular, if the eigenvalues change their signs, then they must become zero
 (see the proof of Proposition~\ref{prop:5b}).
 
Assume that $\D_v f(v^\sigma(\xi);|\chi^\sigma(\xi)|)$
 has a simple zero eigenvalue at $\xi=\xi_*\in(0,1)$ with $\chi^\sigma(\xi_*)\neq 0$.
Let $\bar{e}\in\Rset^{2n_0}$ denote the associated eigenvector.
Then the equilibrium $v^\sigma(\xi_*)$ has a one-dimensional center manifold \cite{GH83,W03}
 on which the system \eqref{eqn:dsysor} reduces to
\begin{equation}
\dot{v}_\c=c_1v_\c^j+c_2(\beta-\beta_*)+O(v_\c^{j+1}),\quad
v_\c\in\Rset,
\label{eqn:lem6a}
\end{equation}
where $\beta_*=|\chi^\sigma(\xi_*)|$, the $O(v_\c^{j+1})$-terms are independent of $\beta$,}
 $c_1,c_2\in\Rset$ are constants and $j>1$ is an even integer
 since the eigenvalue does not change its sign if $j$ is odd.
{\color{black}Note that
\begin{equation}
\frac{\partial f_i}{\partial\beta}(v;\beta)
=\begin{cases}
(i-n_0-1)/n & \mbox{for $i\le n_0$};\\
(i-n_0)/n  & \mbox{for $n_0<i\le 2n_0$}
\end{cases}
\label{eqn:dfb}
\end{equation}
for $i\in[2n_0]$,
 where $f_i(v;\beta)$ is the $i$th element of $f(v;\beta)$.}
If $\bar{e}$ is linearly independent of $(\d v^\sigma/\d\xi)(\xi_*)$,
 then by \eqref{eqn:lem6a}
 there exists a different family of equilibria from $v^\sigma(\xi)$ near $\xi=\xi_*$,
 but both coincide at $\xi=\xi_*$.
This contradicts that the equilibrium $v^\sigma(\xi)$ is isolated for $\xi\in(0,1)$
 as stated in Remark~\ref{rmk:5b}.
Hence, we can take $\bar{e}=(\d v^\sigma/\d\xi)(\xi_*)$, so that
\[
\frac{\d\chi^\sigma}{\d\xi}(\xi_*)=0
\]
since $\chi^\sigma(\xi_*)\neq0$ and $f\left(v^\sigma(\xi);|\chi^\sigma(\xi)|\right)=0$, so that
\begin{align*}
&
\frac{\d}{\d\xi}f\left(v^\sigma(\xi);|\chi^\sigma(\xi)|\right)\Bigl|_{\xi=\xi_*}\\
&
=\D_vf(v^\sigma(\xi_*);\beta_*)\frac{\d v^\sigma}{\d\xi}(\xi_*)
\mp\frac{\partial f}{\partial\beta}(v^\sigma(\xi_*);\beta_*)\frac{\d\chi^\sigma}{\d\xi}(\xi_*)=0,
\end{align*}
where the upper or lower sign is taken,
 depending on whether $\chi^\sigma(\xi_*)$ is positive or negative.
Similarly, we can show that
 $\D_v f(v^\sigma(\xi_*);\beta_*)$ never has a non-simple zero eigenvalue
 even when $(\d\chi^\sigma/\d\xi)(\xi_*)=0$,
 since the equilibrium $v^\sigma(\xi_*)$ is isolated for $\xi\in(0,1)$.
Thus, we obtain the desired result.
\end{proof}

We turn to the proof of Theorem~\ref{thm:6a}.
As in the proof of Proposition~\ref{prop:5b},
 we first analyze the KM \eqref{eqn:dsyso}.

Fix $\sigma\in\Sigma_{n_0}$, 
 and let $n_+$ and $n_-$, respectively,
 denote the numbers of $\sigma_i=1$ and $-1$, $i \in[2n_0]$.
Take the limit $\xi=n_0a/nK|\hat{C}_\D^\sigma|\to0$.
Then $K\to\infty$ and $\phi_i\to 0$, 
 so that by \eqref{eqn:dsolor}
\[
v_i^\sigma\to\begin{cases}
0 & \mbox{if $\sigma_i=1$;}\\
\mbox{$\pi$ or $-\pi$} & \mbox{if $\sigma_i=-1$}.
\end{cases}
\]
Hence,
\[
\cos v_i\to
\begin{cases}
1 & \mbox{if $\sigma_i=1$;}\\
-1 & \mbox{if $\sigma_i=-1$,}
\end{cases}
\quad
\cos(v_j-v_i)\to
\begin{cases}
1 & \mbox{if $\sigma_i=\sigma_j$;}\\
-1 & \mbox{if $\sigma_i\neq\sigma_j$}
\end{cases}
\]
and
\[
\cos v_i+\sum_{j=1,j\neq i}^{2n_0}\cos(v_j-v_i)\to
\begin{cases}
n_+-n_- & \mbox{if $\sigma_i=1$;}\\
n_--n_+-2  & \mbox{if $\sigma_i=-1$.}
\end{cases}
\]
We appropriately replace the rows and columns
 to write the Jacobian matrix $A$ for the vector field of \eqref{eqn:dsyso} as
\[
{\color{black}
A=\frac{K}{n}A_0+O(1),\quad
K\to\infty,}
\]
as $\xi\to 0$, where
\[
A_0=
\begin{array}{r@{}l}
& \begin{array}{ccccccc}
\mbox{ } & \makebox[1em]{ } & \makebox[2em]{ } & n_++1 & n_++2
\end{array}\\
\begin{array}{l}
\\
n_++1\\
n_++2
\end{array} &
\left(
\begin{array}{ccccccc}
2\hat{n} & 1 & \cdots & 1 & -1 & \cdots & -1\\
1 & 2\hat{n} & & 1 & -1 & \cdots & -1\\
\vdots & &\ddots & \vdots \\
1 & 1 && 2\hat{n} & -1 & \cdots & -1\\
-1 & -1 && 1 & -2(\hat{n}-1) & \cdots & 1\\
\vdots  & \vdots && \vdots && \ddots \\
-1 & -1 && 1 &1 && -2(\hat{n}-1)
\end{array}
\right)
\end{array}
\]
with $\hat{n}=n_0-n_+=n_--n_0$.
See Eqs.~\eqref{eqn:A1} and \eqref{eqn:A2}.

\begin{lem}
\label{lem:6b}
The numbers of zero and positive eigenvalues of the matrix $A_0$
 are one and $\min(n_-,n_++1)$, respectively.
In particular, the matrix $A_0$ has no positive eigenvalue if and only if $n_-=0$, i.e., $n_+=2n_0$.
\end{lem}

\begin{proof}
Let $e_j\in\Rset^n$ be the vector of which the $j$th element is one and the others are zero,
 for $j=1,\dots,n$.
We easily see that $\sum_{j=1}^ne_j$ is an eigenvector for the zero eigenvalue.
Moreover, we can show the following:
\begin{enumerate}
\setlength{\leftskip}{-1.8em}
\item[(a)] Case of $n_+=2n_0$ ($n_-=0$):\\
$e_1-e_j$, $j=2,\ldots, n$, are $n-1$ eigenvectors for the eigenvalue $-n=-(2n_0+1)$. 
\item[(b)] Case of $n_+=2n_0-1$  ($n_-=1$):\\
$e_1+\ldots+e_{n-1}-(n-1)e_n$
 is an eigenvector for the eigenvalue $n$.
Moreover, $e_1-e_j$, $j=2,\ldots,n-1$, are $n-2$ eigenvectors
 for the eigenvalue $-(n-2)=-(2n_0-1)$.
{In particular, the matrix $A_0$ has only one positive eigenvalue.}
\item[(c)] Case of $0<n_+<2n_0-1$  ($1<n_-<2n_0$):\\
$n_-(e_1+\ldots+e_{n_++1})-(n_++1)(e_{n_++2}+\cdots+e_n)$
 is an eigenvector for the eigenvalue $n$.
Moreover, $e_1-e_j$, $j=2,\ldots,n_++1$, are $n_+$ eigenvectors for the eigenvalue $2\hat{n}-1$
 and $e_{n_++2}-e_j$, $j=n_++3,\ldots,n$, are $n_--1$ eigenvectors
 for the eigenvalue $-2\hat{n}+1$.
Only either the eigenvalue $-2\hat{n}+1$ or $2\hat{n}-1$ is positive,
 depending on whether $n_+\ge n_0$ or not.
Thus, the matrix $A_0$ has $\min(n_-,n_++1)$ positive eigenvalues.
\item[(d)] Case of $n_+=0$  ($n_-=2n_0$):\\
$(n-1)e_1-(e_2+\cdots+e_n)$ is an eigenvector for the eigenvalue $n$.
Moreover, $e_2-e_j$, $j=3,\ldots,n$, are $n-2$ eigenvectors
 for the eigenvalue $-(n-2)$.
{In particular, the matrix $A_0$ has only one positive eigenvalue.}
\end{enumerate}
Obviously, $\min(n_-,n_++1)=1$ in cases (b) and (d) while it is zero in case (a).
Hence, we obtain the desired result.
\end{proof}

When $\xi=n_0a/nK|\hat{C}_\D^\sigma|\approx 0$,
 the conclusion of Lemma~\ref{lem:6b} on the numbers of zero and positive eigenvalues
 also hold for $A$.
So, when $\sigma_i=1$, $i\in[2n_0]$,
 the matrix $A$ has no positive eigenvalue for $\xi\approx 0$.
Since $\d\chi^\sigma(\xi)/\d\xi=0$ only at $\xi=\xi_0$,
 as shown in the proof of Theorem~\ref{thm:5a}(i),
 we obtain part~(i) by Lemma~\ref{lem:6a}.
Note that at $\xi=\xi_0$ the matrix $A$ keeps the zero eigenvalue
 and one of the negative eigenvalues becomes positive
 at $\xi=\xi_\ast$ when $\xi$ is increased from zero.

We turn to the proof of part~(ii).
{\color{black}
Let $\sigma\in\Sigma_{n_0}$ be fixed and let $\beta=a/K$,
 as in the proof of Lemma~\ref{lem:6a}.
We also write $v^\sigma(\xi)=v^\sigma\in\Tset^{2n_0}$
 and $f(v;\beta)$ for the the vector field of \eqref{eqn:dsysor}
 when the time is changed as $t\to t/K$.
When $\beta=|\chi^\sigma(\xi)|\neq 0$,
 we} compute
\[
\frac{\d v_i^\sigma}{\d\xi}(\xi)
=\begin{cases}
\displaystyle\pm\sigma_i\left(\frac{i-n_0-1}{n_0}\right)\bigg/
 \sqrt{1-\left(\frac{i-n_0-1}{n_0}\xi\right)^2} & \mbox{for $i\le n_0$};\\[2ex]
\displaystyle\pm\sigma_i\left(\frac{i-n_0}{n_0}\right)\bigg/
 \sqrt{1-\left(\frac{i-n_0}{n_0}\xi\right)^2}  & \mbox{for $n_0<i\le 2n_0$}
\end{cases}
\]
for $i\in[2n_0]$,
 where $v_i^\sigma(\xi)$ is the $i$th element of $v^\sigma(\xi)$
 and  the upper or lower sign is taken,
 depending on whether $\chi^\sigma(\xi)$ is positive or not.

{\color{black}
Suppose that $\D_vf(v^\sigma(\xi);|\chi^\sigma(\xi_*)|)$
 has a simple zero eigenvalue at $\xi=\xi_*$ and let $\beta_*=|\chi^\sigma(\xi_*)|\neq 0$.
By Lemma~\ref{lem:6a},
 $(\d v/\d\xi)(\xi_*)$ is the associated eigenvector
 and $(\d\chi^\sigma/\d\xi)(\xi)$ has a zero at $\xi=\xi_*$.
The equilibrium $v^\sigma(\xi_*)$ has a one-dimensional center manifold
 on which the system \eqref{eqn:dsysor} reduces to \eqref{eqn:lem6a}.
Substituting $v=(\d v^\sigma/\d\xi)(\xi_*)v_\c+v^\sigma(\xi_*)$
 into \eqref{eqn:dsysor}
 and taking the inner product of the resulting equation with $(\d v/\d\xi)(\xi_*)$,
 we see that
\[
c_2=\frac{\d v_i^\sigma}{\d\xi}(\xi_*)
 \cdot\frac{\partial f_i}{\partial\beta}(v^\sigma(\xi_*);\beta_*)\bigg/
 \biggl| \frac{\d v_i^\sigma}{\d\xi}(\xi_*)\biggr|^2
\]
is positive or negative,
 depending on whether $h^\sigma(\xi_*)\chi^\sigma(\xi_*)$ is positive or negative, 
 where the dot ``$\cdot$''represents the standard inner product and
\begin{align*}
h^\sigma(\xi)=&\frac{1}{n_0}\left(\sum_{i=1}^{n_0}\sigma_i\left(\frac{i-n_0-1}{n_0}\right)^2\bigg/
 \sqrt{1-\left(\frac{i-n_0-1}{n_0}\xi\right)^2}\right.\\
&
+\left.\sum_{i=n_0+1}^{2n_0}\sigma_i\left(\frac{i-n_0}{n_0}\right)^2\bigg/
 \sqrt{1-\left(\frac{i-n_0}{n_0}\xi\right)^2}\right),
\end{align*}
since by \eqref{eqn:dfb}
\begin{equation}
\frac{\d v_i^\sigma}{\d\xi}(\xi_*)
 \cdot\frac{\partial f_i}{\partial\beta}(v^\sigma(\xi_*);\beta_*)
=\pm\frac{n_0^2}{n}h^\sigma(\xi_*),
\label{eqn:h}
\end{equation}
where the upper or lower sign is taken,
  depending on whether $\chi^\sigma(\xi_*)$ is positive or negative.}
We can determine the sign of $h^\sigma(\xi_*)\chi^\sigma(\xi_*)$ as follows.

\begin{lem}
\label{lem:6c}
{\color{black}If $\chi^\sigma(\xi_*)\neq 0$,
 then} $h^\sigma(\xi_*)\chi^\sigma(\xi_*)>0$.
\end{lem}

\begin{proof}
We easily see that
\begin{align}
\frac{\d\chi^\sigma}{\d\xi}(\xi)
=&\frac{1}{n_0}\left(1+
\sum_{i=1}^{n_0}\sigma_i\left(1-2\left(\frac{i-n_0-1}{n_0}\xi\right)^2\right)\bigg/
 \sqrt{1-\left(\frac{i-n_0-1}{n_0}\xi\right)^2}\right.\notag\\
&
\left.+\sum_{i=n_0+1}^{2n_0}\sigma_i\left(1-2\left(\frac{i-n_0}{n_0}\xi\right)^2\right)\bigg/
 \sqrt{1-\left(\frac{i-n_0}{n_0}\xi\right)^2}\right)\notag\\
=&\frac{\chi^\sigma(\xi)}{\xi}-h^\sigma(\xi)\xi^2
\label{eqn:lem6c}
\end{align}
for $\xi\in(0,1)$, so that
\[
h^\sigma(\xi_*)=\frac{\chi^\sigma(\xi_*)}{\xi_*^3}
\]
{\color{black}since $(\d\chi^\sigma/\d\xi)(\xi_*)=0$.}
This yields the desired result.
\end{proof}

By Lemma~\ref{lem:6c} we see that $c_2>0$.
Using \eqref{eqn:lem6a}, we obtain the following.

{\color{black}
\begin{lem}
\label{lem:6d}
If $|\chi^\sigma(\xi)|$ has a local maximum $($resp. minimum$)$
 at $\xi=\xi_\ast$ with $|\chi^\sigma(\xi_\ast)|\neq 0$,
 then a negative $($resp. positive$)$ eigenvalue
 of $\D_v f(v^\sigma(\xi);|\chi^\sigma(\xi)|)$
 becomes positive $($resp. negative$)$ when $\xi$ increases near $\xi_*$
 while the other nonzero eigenvalues do not change their signs.
\end{lem}

\begin{figure}[t]
\includegraphics[scale=0.8]{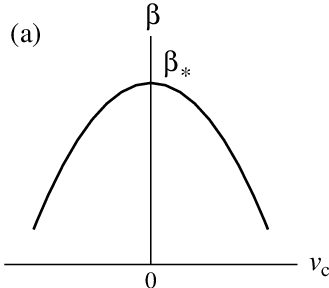}\qquad
\includegraphics[scale=0.8]{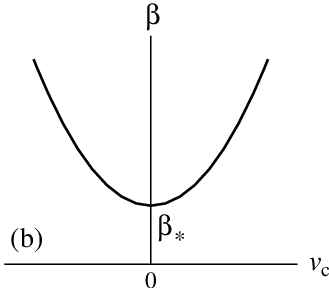}
\caption{Bifurcation diagrams of equilibria for \eqref{eqn:lem6a}:
(a) $\beta=|\chi^\sigma(\xi)|$ has a local maximum; (b) local minimum.
\label{fig:6a}}
\end{figure}

\begin{proof}
Suppose that $|\chi^\sigma(\xi)|$ has a local maximum $($resp. minimum$)$ at $\xi=\xi_*$
It is clear that $c_1$ is positive (resp. negative) in \eqref{eqn:lem6a},
 since $c_2>0$ and by Theorem~\ref{thm:5a}(i)
 a subrcritical (resp. supercritical) saddle-node bifurcation occurs
 when $\beta$ is taken as a control parameter. 
Noting that
\[
v_\c=\frac{\d v^\sigma}{\d\xi}(\xi_*)\cdot(v-v^\sigma(\xi_*))
 \bigg/\left|\frac{\d v^\sigma}{\d\xi}(\xi_*)\right|^2,
\]
we obtain the desired result
 since the derivative of the right-hand side of \eqref{eqn:lem6a}
  with respect to $v_\c$
 at the equilibrium changes from negative to positive
 (resp.  from positive to negative)
 when the sign of $v_\c$ at the equilibrium changes from negative to positive
 (see Fig.~\ref{fig:6a}).
\end{proof}

The graph $\beta=\chi^\sigma(\xi)$ may intersect the zero-axis.
We have the following when such intersection occurs.

\begin{lem}
\label{lem:6e}
Let $\ell_+(\xi)$ and $\ell_-(\xi)$ be, respectively,
 the numbers of positive and negative eigenvalues
 of the Jacobian matrix $\D_v f(v^\sigma(\xi);|\chi^\sigma(\xi)|)$.
If $\chi^\sigma(\xi_*)=0$ for $\xi_*\in(0,1)$,
 then the following hold$\,:$
\begin{enumerate}
\setlength{\leftskip}{-1.8em}
\item[(i)]
$\chi^\sigma(\xi)$ has no extremum at $\xi=\xi_*$,
 i.e., it intersects the zero-axis at $\xi=\xi_*$ transversely$\,;$
\item[(ii)]
When $\xi_1<\xi_*<\xi_2$ such that $(\d\chi^\sigma/\d\xi)(\xi)\neq 0$ on $[\xi_1,\xi_2]$,
 we have $\ell_+(\xi_1)=\ell_-(\xi_2)$ and $\ell_-(\xi_1)=\ell_+(\xi_2)$.
\end{enumerate}
\end{lem}

\begin{proof}
We first prove part~(i) by contradiction.
Assume that $(\d\chi^\sigma/d\xi)(\xi_*)=0$.
Noting that Eq.~\eqref{eqn:lem6c} holds when $\chi^\sigma(\xi_*)=0$,
 we have $h^\sigma(\xi)=0$,
 so that by \eqref{eqn:h} $c_2=0$ in \eqref{eqn:lem6a}.
Since the $O(v_\c^{j+1})$-terms in \eqref{eqn:lem6a} are independent of $\beta$, 
 there is no analytic family of equilibria near $\beta=0$.
This yields a contradiction since $v^\sigma(\xi)$ is such a family of equilibria.
Thus, we obtain part~(i).

On the other hand, 
 by \eqref{eqn:dsolor} and \eqref{eqn:cor4a},
 we can take $\xi_1,\xi_2$ near $\xi_*$ such that $v^\sigma(\xi_1)=-v^\sigma(\xi_2)$.
Then $\D_vf(v^\sigma(\xi_1);|\chi^\sigma(\xi_1)|)=-\D_vf(v^\sigma(\xi_2);|\chi^\sigma(\xi_2)|)$.
This yields part~(ii).
\end{proof}

Now we are in position to complete the proof of Theorem~\ref{thm:6a}(ii).
We see that $|\chi^\sigma(\xi)|$ takes its local maximum before it takes a local minimum,
 when $\xi$ increases from zero,
 since it tends to zero as $\xi\to 0$.
Using Lemmas~\ref{lem:6b} and \ref{lem:6d},
 we show that $\D_v f(v^\sigma(\xi);|\chi^\sigma(\xi)|)$
 keeps one positive eigenvalue at least
 unless $\chi^\sigma(\xi)$ intersects the zero axis or before it does,
 except when $\sigma_i=1$, $i\in[2n_0]$.
By Lemma~\ref{lem:6e} this is true after it intersects the zero axis if it does,
 since $\ell_+(\xi),\ell_-(\xi)\ge 1$ before that, as shown just above.
Thus, we obtain the desired result.}
\end{proof}

\begin{rmk}
In Section~$3$ of {\rm\cite{MS05}},
 the statement of Theorem~{\rm\ref{thm:6a}(ii)} was given in more general setting
 but their proof was not complete.
\end{rmk}

\begin{figure}[t]
\includegraphics[scale=0.58]{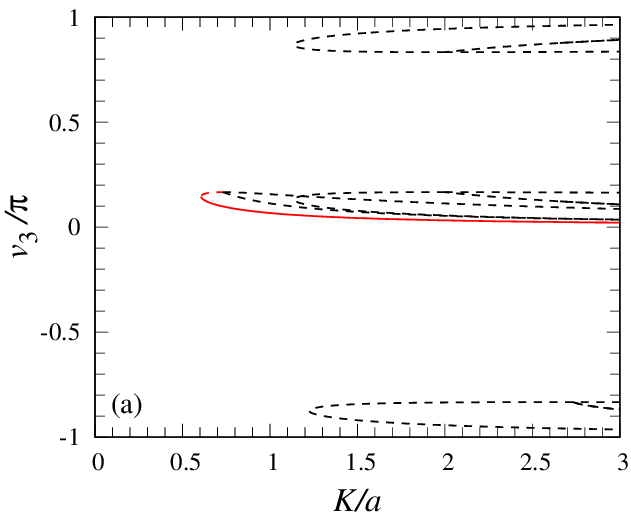}\
\includegraphics[scale=0.58]{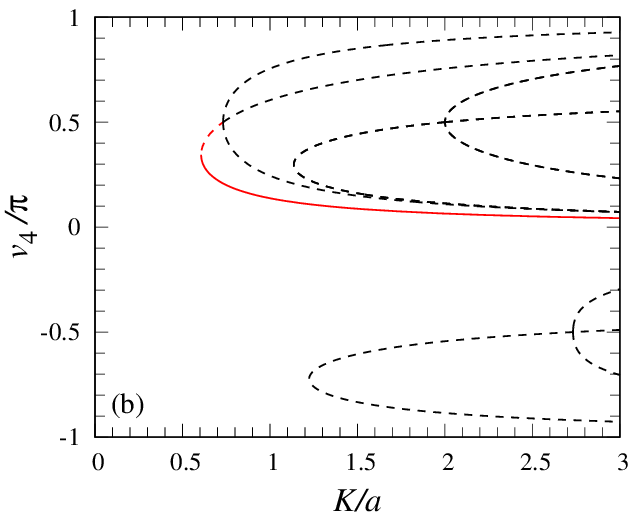}
\caption{Numerically computed bifurcation diagram of equilibria in \eqref{eqn:dsysor} with $n=5$:
(a) $v_3$-component; (b) $v_4$-component.
See the text for more details.
\label{fig:6b}}
\end{figure}

\begin{figure}[t]
\includegraphics[scale=0.58]{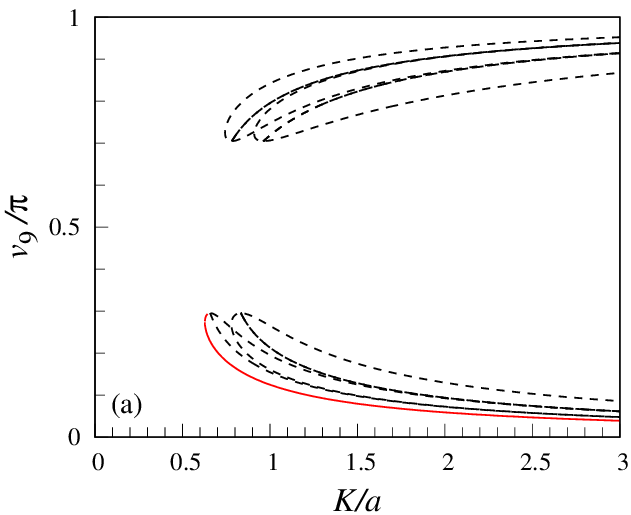}\
\includegraphics[scale=0.58]{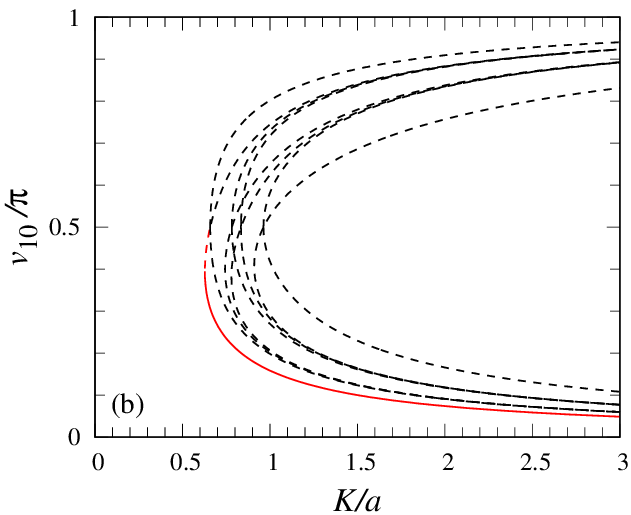}
\caption{Numerically computed bifurcation diagram of equilibria in \eqref{eqn:dsysor} with $n=11$:
(a) $v_9$-component; (b) $v_{10}$-component.
See the text for more details.
\label{fig:6c}}
\end{figure}

Numerically computed bifurcation diagrams of equilibria in \eqref{eqn:dsysor} for $n=5$ and $11$
 are displayed in Figs.~\ref{fig:6b} and \ref{fig:6c}, respectively.
Here the computer tool \texttt{AUTO} \cite{DO12} was used, again.
The solid and dashed lines represent stable and unstable  equilibria, respectively.
The red and black lines, respectively,
 correspond to the equilibrium given by \eqref{eqn:dsolor0} and the others.
They are plotted for all equilibria in Fig.~\ref{fig:6b},
 but only for $(\sigma_2,\sigma_3,\sigma_4,\sigma_5)=(1,1,1,1)$,
\[
(\sigma_6,\sigma_7,\sigma_8,\sigma_9)
 =(1,1,1,1), (1,1,1,-1)\mbox{ or }(1,1,-1,1)
\]
and $\sigma_1,\sigma_{10}=1$ or $-1$ in Fig.~\ref{fig:6c}.
We observe that only the equilibrium with $\sigma_i=1$, $i\in[2n_0]$,
 i.e., given by \eqref{eqn:dsolor0}, can be stable,
 as in the case of $n=3$ and stated in Theorem~\ref{thm:6a}.


\section{Continuum Limits}

Finally we consider the CL \eqref{eqn:csys}
 with the frequency function \eqref{eqn:omegaex}
 and discuss the implications of the above results to it.

The synchronized solution \eqref{eqn:csol} becomes
\begin{equation}
u(t,x)=U(x)+\theta,\quad
U(x)=\arcsin\left(\frac{a(x-\tfrac{1}{2})}{KC}\right),
\label{eqn:csol0}
\end{equation}
where $\theta\in\Sset^1$ is an arbitrary constant
 and the constant $C$ satisfies \eqref{eqn:Cex}.
{\color{black}
We easily see that the CL \eqref{eqn:csys}
 has another family of continuous stationary solutions
\begin{equation}
u(t,x)=\pi-U(x)+\theta
\label{eqn:csol1}
\end{equation}
where $U(x)$ is the same as in \eqref{eqn:csol0}.
From an argument in Section~1,
 we obtain the following result.

\begin{thm}
\label{thm:7a}
The two families \eqref{eqn:csol0} and \eqref{eqn:csol1} give the only continuous stationary solutions
 in the CL \eqref{eqn:csys} with the frequency function \eqref{eqn:omega}
 and they exist if and only if $K/a\ge 2/\pi$.
\end{thm}

We remark that the two family \eqref{eqn:csol0} and \eqref{eqn:csol1}
 do not coincide at $K/a=2/\pi$.}
Theorem~\ref{thm:7a} does not deny the existence of other discontinuous stationary solutions
 to \eqref{eqn:csys}.
Indeed, Eq.~\eqref{eqn:csys} has infinitely many one-parameter families
 of discontinuous stationary solutions of the form
\begin{equation}
u(t,x)=\begin{cases}
U(x)+\theta
 & \mbox{for $x\in[0,1]\setminus\hat{I}$;}\\
\pi-U(x)+\theta & \mbox{for $x\in\hat{I}_j^+$, $j\in[m_+]$;}\\
-U(x)-\pi+\theta & \mbox{for $x\in\hat{I}_j^-$, $j\in[m_-]$,}
\end{cases}
\label{eqn:csol2}
\end{equation}
where $\theta\in\Sset^1$ is an arbitrary constant,
 $m_\pm$ are nonnegative integers and may be infinite,
 $\hat{I}_j^{\pm}\subset[0,1]$, $j\in[m_\pm]$, are intervals
 such that $\hat{I}_j^-\subset[0,\tfrac{1}{2}]$, $\hat{I}_j^+\subset[\tfrac{1}{2},1]$
 and $\hat{I}_j^-\cap \hat{I}_k^-,\hat{I}_j^+\cap \hat{I}_k^+=\emptyset$,
\[
\hat{I}=\bigcup_{j=1}^{m_-}\hat{I}_j^-\cup\bigcup_{j=1}^{m_+}\hat{I}_j^+\neq\emptyset,
\]
and the constant $C$ in $U(x)$ satisfies
\[
C=\int_{[0,1]\setminus\hat{I}}\sqrt{1-\left(\frac{a(x-\tfrac{1}{2})}{KC}\right)^2}\d x
 -\int_{\hat{I}}\sqrt{1-\left(\frac{a(x-\tfrac{1}{2})}{KC}\right)^2}\d x.
\]
Here $U(x)\ge 0$ on $I_j^+$, $j\in[m_+]$,
 $U(x)<0$ on $I_j^-$, $j\in[m_-]$,
 and the interiors of $\hat{I}_j^\pm$, $j\in[m_\pm]$, may be empty. 
We see that the one-parameter family of synchronized solutions given by \eqref{eqn:dsyssol}
 converge to \eqref{eqn:csol2} when
\begin{equation}
\sigma_i=\begin{cases}
1 & \mbox{if $i/n\in[0,1]\setminus\hat{I}$;}\\
-1 & \mbox{if $i/n\in\hat{I}$.}
\end{cases}
\label{eqn:sigma1}
\end{equation}
Using Theorems~\ref{thm:2c} and \ref{thm:2d}, we prove the following

{\color{black}
\begin{thm}\
\label{thm:7b}
\begin{enumerate}
\setlength{\leftskip}{-1.8em}
\item[(i)]
The family of continuous stationary solutions given by \eqref{eqn:csol0} is asymptotically stable,
 while the family of continuous stationary solutions given by \eqref{eqn:csol1}
 is unstable.
\item[(ii)]
The family of discontinuous stationary solutions given by \eqref{eqn:csol2} is unstable
 if it is different from \eqref{eqn:csol0} in the sense of $L^2(I)$.
\end{enumerate}
\end{thm}}
 
\begin{proof}
We easily see that the family of synchronized solutions \eqref{eqn:dsyssol}
 for $v=v^\sigma$ with $\sigma_i=1$, $i\in[2n_0]$,
 which is asymptotically stable for $\xi<\xi_0$ by Theorem~\ref{thm:6a}(i),
 in the KM \eqref{eqn:dsys}
 converges to the family of continuous stationary solutions \eqref{eqn:csol0} as $n\to\infty$.
Recall  that $\chi^\sigma(\xi)$  has a unique extremum (local maximum) at $\xi=\xi_0$
 (see Theorem~\ref{thm:5a}(i)).
Moreover, since
\[
\chi^\sigma(\xi)\to\frac{n}{n_0}\xi\int_0^1\sqrt{1-\xi^2x^2}\,\d x
=\frac{n}{2n_0}\bigl(\xi\sqrt{1-\xi^2}+\arcsin\xi\bigr)
=\frac{n}{2n_0}\varphi(\xi),
\]
we have $\xi_0\to 1$ as $n\to\infty$.
{\color{black}
On the other hand, the family of synchronized solutions \eqref{eqn:dsyssol}
 for $v=v^\sigma$ with $\sigma_i=-1$, $i\in[2n_0]$,
 which is unstable by Theorem~\ref{thm:6a}(ii),
 in the KM \eqref{eqn:dsys}
 converges to the family of continuous stationary solutions \eqref{eqn:csol1} as $n\to\infty$.
Using Theorems~\ref{thm:2c}(i) and \ref{thm:2d}(i),} we obtain part~(i).

We turn to the proof of part~(ii).
The family of synchronized solutions \eqref{eqn:dsyssol}
 for $v=v^\sigma$ with \eqref{eqn:sigma1},
 which is unstable by Theorem~\ref{thm:6a}(ii),
 converges to the family of discontinuous stationary solutions \eqref{eqn:csol2} in $L^2(I)$
 as $n\to\infty$.
Hence, by Theorem~\ref{thm:2d}(i),
 if it is different from \eqref{eqn:csol0} in the sense of $L^2(I)$,
 then the family \eqref{eqn:csol2} is unstable.
\end{proof}

From Theorems~\ref{thm:5a} and \ref{thm:6a} we obtain the following result
 for the KM \eqref{eqn:dsys}.
The family of synchronized solutions \eqref{eqn:dsyssol} for $v^\sigma$ with $\sigma_i=1$, $i\in[2n_0]$,
 suffers a saddle-node bifurcation at $\xi=\xi_0$, i.e., $K=n_0a/n\chi^\sigma(\xi_0)$
 (see Eq.~\eqref{eqn:con}),
 and it turns unstable from stable.
In addition, it suffers a pitchfork bifurcation at $\xi=1$, i.e., $K=n_0a/n\chi^\sigma(1)$,
 and it changes to the family
 with $\sigma_1=\sigma_{2n_0}=-1$ and $\sigma_j=1$, $i\in[2n_0]\setminus\{1,n_0\}$,
 while two families with $\sigma_1=-1$, $\sigma_{2n_0}=1$ or $\sigma_1=1$, $\sigma_{2n_0}=-1$
 and $\sigma_j=1$, $i\in[2n_0]\setminus\{1,n_0\}$, are born.
When $n\to\infty$, the saddle-node and pitchfork bifurcations collide,
 since $\xi_0\to 1$ as shown in the proof of Theorem~\ref{thm:7b}.
Moreover, these four families converge to stationary solutions
 that are the same in the sense of $L^2(I)$ and asymptotically stable
 as a family of $L^2(I)$ solutions by Theorem~\ref{thm:7b}.
More generally, if $\sigma_i=1$
 except for a  fixed finite positive number of $i\in[2n_0]$,
 then by Theorem~\ref{thm:6a}(ii)
 the family of synchronized solutions \eqref{eqn:dsyssol} for $v^\sigma$ is unstable
 but converges to the asymptotically stable family \eqref{eqn:csol0} in the sense of $L^2(I)$
 as $n\to\infty$.
Thus, bifurcation behavior in the CL \eqref{eqn:csys} is very subtle,
 compared with finite-dimensional dynamical systems such as  the KM \eqref{eqn:dsys}.

\section*{Acknowledgements}
This work was partially supported by the JSPS KAKENHI Grant Number JP23K22409.
The author thanks the anonymous referees
 for their helpful comments and constructive suggestions,
 which have improved this work.


\appendix

\renewcommand{\theequation}{\Alph{section}.\arabic{equation}}

\section{Stability of the equilibria in \eqref{eqn:dsys3r}}

Let
\[
z=\sqrt{1-\left(\frac{a}{K}\right)^2}\in(0,1).
\]
For \eqref{eqn:dsol3b}, we have
\[
{-\cos v_1-\cos v_2-\cos(v_2-v_1)=1>0}
\]
and
\[
\cos v_1\cos v_2+(\cos v_1+\cos v_2)\cos(v_2-v_1)
 =-z^2<0.
\]
{By \eqref{eqn:u3}}
 the equilibria given by \eqref{eqn:dsol3b} are unstable.

Let
\[
z=\sqrt{1-\left(\frac{a}{3KC_\D}\right)^2}\in(0,1).
\]
For \eqref{eqn:dsol3r1}, we have
{
\[
-\cos v_1-\cos v_2-\cos(v_2-v_1)
 =-2z^2-2z+1=-2(z+\tfrac{1}{2})^2+\tfrac{3}{2}
\]
which is negative when $z>(\sqrt{3}-1)/2$,} and
\[
\cos v_1\cos v_2+(\cos v_1+\cos v_2)\cos(v_2-v_1)
 =4z^3+z^2-2z,
\]
which is {
 positive or negative,
 depending on whether $z$ is greater or less than $(\sqrt{33}-1)/8>(\sqrt{3}-1)/2$.}
By \eqref{eqn:s3} and \eqref{eqn:u3},
 the equilibria given by \eqref{eqn:dsol3r1} are asymptotically stable (resp. unstable)
 if $v_2=-v_1$ is less (resp. greater) than $v_0$.

Let
\[
z=\sqrt{1-\left(\frac{a}{3K\hat{C}_\D}\right)^2}\in(0,1).
\]
For \eqref{eqn:dsol3r2}, we have
{
\[
-\cos v_1-\cos v_2-\cos(v_2-v_1)
 =-2z^2+2z+1=-2(z-\tfrac{1}{2})^2+\tfrac{3}{2}>0.
\]
By \eqref{eqn:u3},}
 the equilibria given by \eqref{eqn:dsol3r2} are unstable.

\section{Line Colors  in Figs.~\ref{fig:5a}(b) and (c)}

In this appendix,
 we provide the meaning of line colors in Figs.~\ref{fig:5a}(b) and (c),
 where $n=11$ and $n_0=5$.
For each line color, $\sigma$ is specified for the graphs of the equilibria
 below in the same order as the corresponding graphs in Fig.~\ref{fig:5a}(b) or (c).

\subsection{Fig.~\ref{fig:5a}(b)}

For $\sigma\in\Sigma_5$ in Fig.~\ref{fig:5a}(b),
 we have $\sigma_{11-i}=\sigma_i$ for any $i\neq 1$
 and $(\sigma_1,\sigma_{10})=(1,1)$, $(-1,-1)$, $(-1,1)$ or $(1,-1)$.
So we only give the values of $\sigma_i$ for $i\in[5]\setminus\{1\}$ below.

{\color{black}
\begin{itemize}
\setlength{\leftskip}{-2.6em}
\item
Black lines:
\[
(\sigma_2,\sigma_3,\sigma_4,\sigma_5)=(1,1,1,1), (-1,1,-1,1), (-1,-1,1,-1).
\]
\item
Red lines:
\[
(\sigma_2,\sigma_3,\sigma_4,\sigma_5)=(-1,1,1,1), (-1,1,1,-1), (-1,1,-1,-1).
 \]
\item
Blue lines:
\[
(\sigma_2,\sigma_3,\sigma_4,\sigma_5)=(1,-1,1,1), (1,-1,-1,1), (1,-1,-1,-1).
\]
\item
Purple lines:
\[
(\sigma_2,\sigma_3,\sigma_4,\sigma_5)=(1,1,-1,1), (1,-1,1,-1), (-1,-1,-1,-1).
\]
\item
Green lines:
\[
(\sigma_2,\sigma_3,\sigma_4,\sigma_5)=(1,1,1,-1), (1,1,-1,-1).
\]
\item
Orange lines:
\[
(\sigma_2,\sigma_3,\sigma_4,\sigma_5)=(-1,-1,1,1), (-1,-1,-1,1).
\]
\end{itemize}}

\subsection{Fig.~\ref{fig:5a}(c)}
For $\sigma\in\Sigma_5$ in Fig.~\ref{fig:5a}(c),
 we have $\sigma_i\neq\sigma_{11-i}$ for some $i\neq 1$
 and $(\sigma_1,\sigma_{10})=(1,1)$, $(-1,-1)$, $(-1,1)$ or $(1,-1)$.
We only give the values of $\sigma_i$ for $i\in[5]\setminus\{1\}$ below,
 where $\sigma_i\neq\sigma_{11-i}$ if $\sigma_i=\pm 1$
 and $\sigma_i=\sigma_{11-i}$ otherwise.

{\color{black}
\begin{itemize}
\setlength{\leftskip}{-2.6em}
\item
Black lines:
\begin{align*}
(\sigma_2,\sigma_3,\sigma_4,\sigma_5)
=& (\pm 1,1,1,1), (\pm 1,1,1,\pm 1), (-1,1,1,\pm 1),\\
& (1,-1,\pm 1,1), (1,1,-1,\pm 1), (\pm 1,-1,1,\pm1),\\
& (1,-1,\pm 1,\pm 1), (-1,\pm 1,\pm 1,\pm 1), (\pm 1,-1,1,-1),\\
& (1,\pm 1,-1,-1), (\pm 1,\pm 1,-1,-1).
\end{align*}
\item
Red lines:
\begin{align*}
(\sigma_2,\sigma_3,\sigma_4,\sigma_5)
=& (1,\pm 1,1,1), (1,\pm 1,\pm 1,1), (\pm1,-1,1,1),\\
& (\pm1,1,1,-1), (1,1,\pm 1,-1), (\pm 1,\pm 1,-1,1),\\
& (1,\pm 1,-1,\pm 1), (-1,\pm 1,1,-1), (\pm 1,\pm 1,-1,\pm 1),\\
& (-1,-1,\pm 1,\pm 1), (-1,-1,-1,\pm 1).
\end{align*}
\item
Blue lines:
\begin{align*}
(\sigma_2,\sigma_3,\sigma_4,\sigma_5)
=& (1,1,\pm 1,1), (1,\pm 1,1,\pm 1), (\pm1,\pm 1,\pm 1,1),\\
& (1,-1,1,\pm 1), (-1,\pm 1,\pm 1,1), (\pm1,\pm 1,\pm 1,\pm 1),\\
& (1,\pm 1,\pm 1,-1),(-1,1,-1,\pm 1), (\pm 1,\pm 1,\pm 1,-1),\\
& (-1,\pm 1,-1,\pm 1), (-1,-1,\pm 1,-1).
\end{align*}
\item
Purple lines:
\begin{align*}
(\sigma_2,\sigma_3,\sigma_4,\sigma_5)
=& (1,1,1,\pm 1), (1,1,\pm 1,\pm 1), (\pm 1,\pm 1,1,\pm 1),\\
& (1,\pm 1,-1,1), (-1,\pm 1,1,\pm 1), (\pm1,\pm 1,1,-1),\\
& (-1,-1,\pm 1,1), (\pm 1,-1,-1,1), (\pm 1,1,-1,-1),\\
& (-1,\pm 1,\pm 1,-1), (-1,\pm 1,-1,-1).
\end{align*}
\item
Green lines:
\begin{align*}
(\sigma_2,\sigma_3,\sigma_4,\sigma_5)
=& (\pm 1,\pm 1,1,1), (-1,\pm 1,1,1), (\pm 1,1,-1,1),\\
& (1,\pm 1,\pm 1,\pm 1), (-1,1,\pm 1,\pm 1), (\pm 1,1,-1,\pm 1),\\
& (-1,-1,1,\pm 1), (-1,1,\pm 1,-1), (1,-1,-1,\pm 1), \\
& (\pm 1,-1,-1,\pm 1)
\end{align*}
\item
Orange lines:
\begin{align*}
(\sigma_2,\sigma_3,\sigma_4,\sigma_5)
=& (\pm 1,1,\pm 1,1), (-1,1,\pm 1,1), (\pm 1,1,\pm 1,\pm 1),\\
& (1,\pm 1,1,-1), (\pm 1,-1,\pm 1,1), (\pm 1,1,\pm 1,-1),\\
& (-1,\pm 1,-1,1),(\pm1,-1,\pm 1,\pm 1), (1,-1,\pm 1,-1),\\
& (\pm 1,-1,\pm 1,-1), (\pm 1,-1,-1,-1).
\end{align*}
\end{itemize}}


\end{document}